\documentclass{article}
\usepackage[utf8]{inputenc}
\usepackage{amsmath, amsfonts, amssymb}
\usepackage[dvipsnames]{xcolor}
\usepackage{geometry}
\usepackage{url}
\usepackage{amsthm}
\usepackage{graphicx}
\usepackage{mathpazo}
\usepackage{appendix}
\usepackage{authblk}
\newtheorem{theorem}{Theorem}
\newtheorem{lemma}{Lemma}
\newtheorem{corollary}{Corollary}
\newtheorem{proposition}{Proposition}

\usepackage{hyperref}
\hypersetup{
    colorlinks=true,
    linkcolor= Maroon,
    citecolor= Sepia, 
    urlcolor=MidnightBlue,
    pdfpagemode=FullScreen,
    }

\newcommand{\ee}{\mathrm{e}}

\newcommand{\dd}{\mathrm{d}}

\newcommand{\ii}{\mathrm{i}}

\newcommand{\V}{\mathcal{V}}

\newcommand{\W}{\mathcal{W}}

\begin{document}

\title{Rigorous Methods for Bohr-Sommerfeld Quantization Rules}
\author{J. Dong\footnote{\texttt{jtdong@umich.edu}.}, P. D. Miller\footnote{\texttt{millerpd@umich.edu}.  P. D. Miller was partially supported by the National Science Foundation through grant DMS-2204896.}, and G. Young\footnote{\texttt{gfyoung@umich.edu}.  G. F. Young was partially supported by the National Science Foundation through grant DMS–2303363.}}
\affil{Department of Mathematics, University of Michigan}
\date{May 4, 2025}

\maketitle

\begin{abstract}
    In this work, we prove Bohr-Sommerfeld quantization rules for the self-adjoint Zakharov-Shabat system and the Schr\"odinger equation in the presence of two simple turning points bounding a classically allowed region. In particular, we use the method of comparison equations for $2\times 2$ traceless first-order systems to provide a unified perspective that yields similar proofs in each setting.  The use of a Weber model system gives results that are uniform in the eigenvalue parameter over the whole range from the bottom of the potential well up to finite values.
\end{abstract}

\section{Introduction}
\subsection{Brief history of the WKBJ method and the connection problem}
In 1857, Stokes \cite{Stokes1847} sought to recast Airy's integral $\operatorname{Ai}(z)$ into a form that would ease calculation of its numerical value for large $z$. He noticed that arbitrary constants which appear in the linear combination of two independent power series solutions seemed to be discontinuous across certain values of $\arg (z)$. This paradox of obtaining a discontinuous asymptotic expression for a continuous function came to be known as the \textit{Stokes Phenomenon}. In a subsequent paper \cite{Stokes1850}, Stokes studied this problem in more detail and found that these illusory continuities were due to exponentially small terms becoming suddenly dominant as $\arg (z)$ is varied. In a famous quote from a later 1902 review paper \cite{Stokes1902}, Stokes describes: ``The inferior term enters as it were into a mist, is hidden for a little from view, and comes out with its coefficient changed.''

Gans' 1915 paper \cite{Gans1915} in which he studies the propagation of light through a slowly varying medium is widely considered to be the first systematic study of the behavior of waves on either side of a single transition point. He linearly approximated the coefficient vanishing at the transition point and constructed asymptotic expressions valid on either side, patching approximate solutions in an overlapping region containing the transition point. However, his result was not easily adaptable to other contexts.

In 1925, while studying Mathieu's equation in the modeling of free oscillations of water in an elliptical lake, H. Jeffreys became interested in the effects of transition points on solutions. Before Gans, Horn in 1899 \cite{Horn1899} had derived asymptotic formul\ae\  for linear, first-order equations with a large parameter in the absence of transition points. H. Jeffreys felt that Horn's formul\ae\  were not easily adaptable to handle a transition point, nor did they apply to second-order equations, so he set out to address these issues in a series of three 1925 papers \cite{Jeffreys1925-I,Jeffreys1925-II,Jeffreys1925-III}.  
The following year, without knowledge of Jeffreys' result, Wentzel \cite{Wentzel26}, Kramers \cite{Kramers26}, and Brillouin \cite{Brillouin26} independently derived similar formul\ae\  for Schr\"{o}dinger's equation. Ultimately, the WKB method is named after them, since it became popular for its application in quantum mechanics. We decide to include the ``J'' as it sometimes appears in the literature to respectfully acknowledge Jeffreys' significant contribution. Also in the quantum problem, transition points have the interpretation of turning points for a classical particle, so we will refer to them as such going forward.

In 1929, Zwaan \cite{Zwaan1929} sought  
to avoid turning points altogether in the connection problem by moving into the complex plane. This procedure is motivated by the assumption that within a region containing a turning point, the coefficient vanishing there can be approximated within some error by an analytic function. However, as Langer points out in his review \cite{Langer34}, it is nearly impossible to estimate the error resulting from such an approximation, and the method fails to capture the character of the solution on the real line near the turning point. 

To remedy this, Langer introduced his own approach \cite{Langer34,Langer35,Langer37}. Considering the one-dimensional Schr\"{o}dinger equation, Langer had the idea to map the equation to an approximation of Bessel's equation of order $\nu = (\mu + 2)^{-1}$ in a neighborhood of a turning point of order $\mu$ by combining a linear transformation of the unknown function with a change of the independent variable. Most notable about Langer's technique is that it produces a representation of the solution in a (real) neighborhood of the turning point, whereas previous studies of the connection problem required meticulous patching of approximations valid only on either side  that blow up at the turning point itself. A further generalization of Langer's method is described in Section~\ref{sec:comparison-equations} below.

\subsection{The WKBJ method}
\label{sec:general-system}
In its original form, the famous WKBJ method is a technique for approximating solutions of the Schr\"odinger equation 
    \begin{equation}
        -\frac{\hbar^2}{2m}\psi''(x)+\V(x)\psi(x)=\lambda\psi(x),
        \label{eq:Schrodinger}
    \end{equation}
in the limit of small Planck constant $\hbar$, assuming that the difference $\V(x)-\lambda$ is of a fixed sign.  The classical method is based on the exponential substitution 
\begin{equation}
    \psi(x)=\exp\left(\int^x u(y)\,\dd y\right)
\end{equation}
from which it follows that $u(x)$ satisfies a Riccati equation, 
\begin{equation}
    u'(x)+u(x)^2 + \frac{2m}{\hbar^2}(\lambda-\V(x))=0.
    \label{eq:Eikonal}
\end{equation}
One then develops $u(x)$ in a series of ascending powers of $\hbar$:  $u=\hbar^{-1}u_0(x)+u_1(x)+\hbar u_2(x)+\cdots$.  The leading coefficient satisfies an algebraic equation $u_0(x)^2+2m(\lambda-\V(x))=0$ sometimes called the \emph{eikonal equation}.  The method is especially clearly explained in the paper of Wentzel \cite{Wentzel26}.

In this article, we are concerned more generally with singularly-perturbed first-order systems of the form
\begin{equation}
    \epsilon\frac{\dd\mathbf{w}}{\dd x}=\mathbf{B}(x;\lambda)\mathbf{w},\quad\mathbf{w}=\mathbf{w}(x)=\mathbf{w}(x;\lambda;\epsilon),\quad x\in\mathbb{R},
    \label{eq:general-system}
\end{equation}
 in which $\mathbf{B}(x;\lambda)$ is a given $2\times 2$ coefficient matrix, $\lambda\in\mathbb{R}$ is a parameter, and $0<\epsilon\ll 1$ is the small parameter introducing the singular perturbation.  The analogue of the WKBJ method for such a system is based on a substitution involving a scalar exponential factor $\ee^{\sigma(x;\lambda)/\epsilon}$:
\begin{equation}
    \mathbf{w}(x)=\ee^{\sigma(x;\lambda)/\epsilon}\mathbf{v}(x)
\end{equation}
leading to the equivalent equation
\begin{equation}
    \epsilon\frac{\dd\mathbf{v}}{\dd x} = (\mathbf{B}(x;\lambda)-\sigma'(x;\lambda)\mathbb{I})\mathbf{v}.
\end{equation}
One then takes $\sigma'(x;\lambda)$ as one of the eigenvalues of $\mathbf{B}(x;\lambda)$ (the characteristic equation for $\sigma'(x;\lambda)$ is the analogue of the eikonal equation in this case) and expands $\mathbf{v}(x)$ in a power series in $\epsilon$ whose leading term is a corresponding eigenvector with a normalization freedom that has to be resolved at the next order.  This is done by enforcing solvability of the linear equation for the next correction with singular matrix $\mathbf{B}(x;\lambda)-\sigma'(x;\lambda)\mathbb{I}$, which yields a linear first-order differential equation that the scalar normalization factor must satisfy as a function of $x$.  The solution of the resulting solvable system then includes an arbitrary multiple of the same eigenvector once again, whose value is determined at the next order in exactly the same way.  Thus, the procedure continues to arbitrary order in $\epsilon$.

\subsection{Example systems}
More concretely, we examine in this paper the following two specific examples of the system \eqref{eq:general-system}.
\begin{itemize}
    \item \emph{The Schr\"odinger equation.}   
    If one takes
    \begin{equation}
        \mathbf{w}(x;\lambda,\epsilon):=\begin{bmatrix}\psi(x;\lambda,\epsilon)\\\epsilon\psi'(x;\lambda,\epsilon)\end{bmatrix},\quad\mathbf{B}(x;\lambda):=\begin{bmatrix}0 & 1\\\V(x)-\lambda & 0\end{bmatrix},\quad \epsilon=\frac{\hbar}{\sqrt{2m}},
        \label{eq:B-Schrodinger}
    \end{equation}
    then the first equation of the system is an identity, while the second yields the stationary Schr\"odinger equation \eqref{eq:Schrodinger} for $\psi$ with potential energy function $\V$ and energy $\lambda$.  To exclude continuous spectrum, 
    one usually assumes that the potential satisfies $\V(x)\to+\infty$ as $x\to\pm\infty$.
  
    \item \emph{The Zakharov-Shabat system.} The self-adjoint Zakharov-Shabat system reads
    \begin{equation}
        \epsilon\frac{\dd \mathbf{v}}{\dd x}=\begin{bmatrix}-\ii\lambda & q_0(x)\\q_0(x)^* & \ii\lambda\end{bmatrix}\mathbf{v},
    \end{equation}
    with potential $q_0(x)=A(x)\ee^{\ii S(x)/\epsilon}$, $A(x)>0$ and $S(x)$ real-valued.
    This problem arises in the solution, by the inverse-scattering method \cite{ZakharovS73}, of the initial-value problem for the defocusing nonlinear Schr\"odinger equation
    \begin{equation}
        \epsilon\frac{\partial q}{\partial t} + \frac{1}{2}\epsilon^2\frac{\partial^2 q}{\partial x^2} - |q|^2q=0,\quad q=q(x,t),\quad q(x,0)=q_0(x).
        \label{eq:nls}
    \end{equation}
    We assume the usual boundary conditions on $q_0(x)$:  that $A(x)\to 1$ and $S(x)\to 0$ as $x\to\pm\infty$.  
    Since $q_0(x)$ depends on the small parameter $\epsilon\ll 1$ via a fast phase factor, to obtain a system of the form \eqref{eq:general-system} one should first remove this phase by a substitution:
    \begin{equation}
        \mathbf{v}(x)=\ee^{\ii S(x)\sigma_3/(2\epsilon)}\mathbf{w}(x) = \begin{bmatrix}\ee^{\ii S(x)/(2\epsilon)} & 0\\0 & \ee^{-\ii S(x)/(2\epsilon)}\end{bmatrix}\mathbf{w}(x).
    \end{equation}
    Thus, one arrives at \eqref{eq:general-system}, in which 
    \begin{equation}
        \mathbf{B}(x;\lambda):=\begin{bmatrix} -\ii\left(\lambda+\frac{1}{2}S'(x)\right) & A(x)\\A(x) & \ii\left(\lambda+\frac{1}{2}S'(x)\right)\end{bmatrix}.
        \label{eq:B-ZS}
    \end{equation}
\end{itemize}
The study of the Zakharov-Shabat system in the small-$\epsilon$ limit is therefore relvant to the analysis of the nonlinear initial-value problem \eqref{eq:nls} in the semiclassical limit \cite{Jin,JinLM99}.

In both of these examples, the matrix $\mathbf{B}(x;\lambda)$ has zero trace.  Although this can always be assumed without loss of generality by replacing $\mathbf{w}(x)$ with an appropriate scalar multiple by a factor of the form $\ee^{\alpha(x;\lambda)/(2\epsilon)}$, where $\alpha(x;\lambda)$ is an antiderivative of $B_{11}(x;\lambda)+B_{22}(x;\lambda)$, it is important to keep in mind that such a substitution can alter the boundary conditions to be imposed on $\mathbf{w}(x)$ as $x\to\pm\infty$.  Also in both systems, the determinant $\det(\mathbf{B}(x;\lambda))$ is real-valued.

\subsection{Turning points} 
The WKBJ method in its original form for the Schr\"odinger equation \eqref{eq:Schrodinger} is well known to fail in neighborhoods of values of $x$ where $\V(x)-\lambda$ vanishes; such points are called \emph{turning points}, since a classical particle with total energy $\lambda$ is necessarily confined to an interval where $\V(x)-\lambda<0$ and typically oscillates between two turning points bounding such a \emph{classically allowed zone}.  The generalization to traceless systems of the form \eqref{eq:general-system} of the condition defining turning points is the condition that the eigenvalues $\sigma'(x;\lambda)$ coalesce, necessarily at $\sigma'=0$, and hence the turning point condition is $\det(\mathbf{B}(x;\lambda))=0$.  For systems where the determinant is real-valued, the analogue of classically allowed zones are intervals of $x$ where $\det(\mathbf{B}(x;\lambda))>0$.

\subsection{Generalizing the WKBJ method as a transformation technique}
\label{sec:comparison-equations}
From one point of view, the WKBJ method is an example of a technique for finding changes of variables to map \eqref{eq:general-system} onto perturbations of certain \emph{model systems} so that the following things are true.
\begin{itemize}
\item The model system retains the form of the system, so in the new variables one again has \eqref{eq:general-system} but with a different coefficient matrix denoted $\widetilde{\mathbf{B}}$.
\item The matrix $\widetilde{\mathbf{B}}$ qualitatively resembles the original coefficient matrix $\mathbf{B}$.  The qualitative resemblance is to be enforced at the level of the determinants of the coefficient matrices.  As will be seen, this is important so that the change of variables is smooth.
\item The model system with coefficient matrix $\widetilde{\mathbf{B}}$ can be solved exactly.
\end{itemize}
The simplest implementation of the technique is the Liouville-Green transformation, which places the traditional WKBJ method on rigorous footing in the absence of turning points.  
Going further, the method introduced by Langer \cite{Langer34,Langer35,Langer37} to obtain expansions of solutions of singularly-perturbed differential equations that are uniformly valid in a full neighborhood of a turning point is a more sophisticated implementation.  The general perspective we follow here was described in the review article of Berry and Mount \cite[Section 4]{BerryM72}, where it is called the \emph{method of comparison equations} and is attributed to Miller and Good \cite{MillerGood1953} and Dingle \cite{Dingle56}. Table I in \cite{Dingle56} is a particularly remarkable scientific contribution. 

Here are several different model systems and their properties.
\subsubsection{Model for $\det(\mathbf{B}(x;\lambda))>0$.}
In intervals of $x$ where $\det(\mathbf{B}(x;\lambda))>0$ we may think that a change of variables might be able to map $\mathbf{B}$ onto a perturbation of a simple matrix with a positive constant determinant, say, $1$.  The simplest model system in this case that is still coupled  has coefficient matrix 
\begin{equation}
    \widetilde{\mathbf{B}}(y):=\begin{bmatrix}0 & 1\\-1 & 0\end{bmatrix}.
\end{equation}
Here we are using $y$ to denote the new independent variable.  The exact solution of the model system $\epsilon\mathbf{u}'(y)=\widetilde{\mathbf{B}}(y)\mathbf{u}(y)$ 
is then simply
\begin{equation}
    \mathbf{u}(y)=C_+\ee^{\ii y/\epsilon}\begin{bmatrix}1 \\ \ii\end{bmatrix}+C_-\ee^{-\ii y/\epsilon}\begin{bmatrix}1 \\ -\ii\end{bmatrix} = A\begin{bmatrix}\cos(y/\epsilon)\\-\sin(y/\epsilon)
    \end{bmatrix}+B\begin{bmatrix} \sin(y/\epsilon)\\ \cos(y/\epsilon)\end{bmatrix},
    \label{eq:model-solution-minus-one}
\end{equation}
where $C_+,C_-$, or alternately $A,B$ are arbitrary constants.  This is an oscillatory solution.

\subsubsection{Model for $\det(\mathbf{B}(x;\lambda))<0$.}
In intervals of $x$ where $\det(\mathbf{B}(x;\lambda))<0$ one might model $\mathbf{B}$ with a simple matrix with negative constant determinant, say, $-1$.  The simplest coupled model system in this case has coefficient matrix
\begin{equation}
    \widetilde{\mathbf{B}}(y):=\begin{bmatrix}0 & 1\\1 & 0\end{bmatrix}.
\end{equation}
The exact solution of the model system $\epsilon\mathbf{u}'(y)=\widetilde{\mathbf{B}}(y)\mathbf{u}(y)$ in this case is
\begin{equation}
\mathbf{u}(y)=C_+\ee^{y/\epsilon}\begin{bmatrix}1\\1\end{bmatrix}+C_-\ee^{-y/\epsilon}\begin{bmatrix}1\\-1\end{bmatrix}=A\begin{bmatrix}\cosh(y/\epsilon)\\\sinh(y/\epsilon)\end{bmatrix} + B\begin{bmatrix}\sinh(y/\epsilon)\\\cosh(y/\epsilon)\end{bmatrix},
\label{eq:model-solution-plus-one}
\end{equation}
for arbitrary constants $C_+,C_-$ or $A,B$.  
Depending on the values of the constants, this is a rapidly exponentially growing or decaying solution.

\subsubsection{Model system for simple turning points}
\label{sec:Airy}
Fix $\lambda\in\mathbb{R}$, and a \emph{simple} turning point $x_0\in\mathbb{R}$, i.e., a simple root $x_0$ of $x\mapsto\det(\mathbf{B}(x;\lambda))$.   
In such a case, it seems that $\det(\mathbf{B}(x;\lambda))$ behaves linearly near its root
at $x=x_0$, so to model this we can consider perhaps the simplest coefficient matrix of a coupled system with a similar property:
\begin{equation}
    \widetilde{\mathbf{B}}(y)=\begin{bmatrix}0 & 1\\y & 0\end{bmatrix}.
\end{equation}
The resulting model system $\epsilon\mathbf{u}'(y)=\widetilde{\mathbf{B}}(y)\mathbf{u}(y)$
can be solved using Airy functions\footnote{Or, as Langer viewed them, Bessel functions of order $\nu=(\mu+2)^{-1}=\frac{1}{3}$ for a turning point of multiplicity $\mu=1$.}, since eliminating $u_2(y;\epsilon)$ via $u_2(y;\epsilon)=\epsilon u_1'(y;\epsilon)$ yields Airy's equation in the form $\epsilon^2 u_1''(y;\epsilon)-yu_1(y;\epsilon)=0$.  Scaling out $\epsilon$ via $y=\epsilon^{2/3}z$ and writing $\Phi(z)=u_1(y;\epsilon)$ gives the Airy equation \emph{Airy equation} \cite[Chapter 9]{DLMF} in its standard form:  
\begin{equation}
    \Phi''(z)-z\Phi(z)=0. 
    \label{eq:AiryODE}
\end{equation} 

A fundamental pair of real-valued solutions of \eqref{eq:AiryODE} is denoted $\Phi(z)=\mathrm{Ai}(z),\mathrm{Bi}(z)$.  
The corresponding general solution of the model system is
\begin{equation}
    \mathbf{u}(y)=A\begin{bmatrix}\mathrm{Ai}(\epsilon^{2/3}y)\\\epsilon^{5/3}\mathrm{Ai}'(\epsilon^{2/3}y)\end{bmatrix}+
    B\begin{bmatrix}\mathrm{Bi}(\epsilon^{2/3}y)\\\epsilon^{5/3}\mathrm{Bi}'(\epsilon^{2/3}y)\end{bmatrix},
\end{equation}
where $A,B$ are arbitrary constants.  This solution has a transitional character, with oscillations for $y<0$ and exponential behavior for $y>0$.

\subsubsection{Model system for a pair of simple turning points}\label{section_Weber_equation}
Again fixing $\lambda\in\mathbb{R}$, suppose now that the system \eqref{eq:general-system} has exactly two simple turning points in some interval of $x$, and that $\det(\mathbf{B}(x;\lambda))$ is positive between the turning points.  Perhaps the simplest coefficient matrix that models this behavior without decoupling the system is
\begin{equation}
    \widetilde{\mathbf{B}}(y;b)=\begin{bmatrix}0 & 1\\b+\frac{1}{4}y^2 & 0\end{bmatrix}
\end{equation}
where $b<0$ is a parameter.  The qualitative property we are aiming to capture is present for any $b<0$, but as will seen below it is essential to include this parameter to allow \eqref{eq:general-system} to be mapped to a perturbation of the model system
\begin{equation}
    \epsilon\frac{\dd\mathbf{u}}{\dd y}=\begin{bmatrix}0 & 1\\ b+\frac{1}{4}y^2 & 0\end{bmatrix}\mathbf{u}
\label{eq:Weber-system}
\end{equation}
by a smooth change of variables.  This system is also solvable in terms of special functions.  Indeed, eliminating $u_2(y;b,\epsilon)=\epsilon u_1(y;b,\epsilon)$ one obtains $\epsilon^2 u_1''(y;b,\epsilon)-(b+y^2/4)u_1(y;b,\epsilon)=0$ which is a form of \emph{Weber's equation} (cf., \cite[Chapter 12]{DLMF}).  Scaling out $\epsilon$ via $y=\epsilon^{1/2}z$ and $b=\epsilon a$, and writing $\Phi(z;a)=u_1(y;b,\epsilon)$ gives the standard form of Weber's equation:  
\begin{equation}
    \Phi''(z)-\left(a+\frac{1}{4}z^2\right)\Phi(z)=0,\quad a<0.
    \label{eq:Weber0}
\end{equation}

Solutions of \eqref{eq:Weber0} are known as \textit{parabolic cylinder functions} and they are all entire functions of $z$ and also of the parameter $a$. A pair of solutions exhibiting exponential dichotomy (sometimes called a \emph{numerically satisfactory pair}) as $z\to+\infty$ is denoted $\Phi(z)=U(a,z),V(a,z)$, which have the asymptotic behavior (for fixed $a\in\mathbb{C}$):
\begin{equation}
\begin{split}
    U(a,z)&=\ee^{-z^{2}/4}z^{-a-1/2}(1+O(z^{-2})),\quad z\to+\infty,\\
    V(a,z)&=\sqrt{\frac{2}{\pi}}\ee^{z^{2}/4}z^{a-1/2}(1+O(z^{-2})),\quad z\to+\infty.
    \end{split}
    \label{eq:UV-asymp}
\end{equation}
These solutions have the (contstant, by Abel's theorem) Wronskian
\begin{equation}
    \W [ U(a,\diamond), V(a,\diamond)](z) = U(a,z)\frac{\partial V}{\partial z}(a,z) - V(a,z)\frac{\partial U}{\partial z}(a,z) = \sqrt{\frac{2}{\pi}}.
    \label{eq:Weber-Wronskian}
\end{equation}
The pair $\Phi(z)=U(a,-z),V(a,-z)$ is a numerically satisfactory fundamental pair of solutions of \eqref{eq:Weber0} as $z\to -\infty$.

\subsubsection{Model systems for more than two simple turning points}
Following the same line of reasoning, it seems that if one needs to consider an interval of $x$ in which the system \eqref{eq:general-system} has three or more turning points, counted with multiplicity, a suitable model system could be
\begin{equation}
    \epsilon\frac{\dd \mathbf{u}}{\dd y}=\begin{bmatrix}0 & 1\\m(y) & 0\end{bmatrix}\mathbf{u}
    \label{eq:higher-order-model}
\end{equation}
where $m(y)$ is a polynomial of degree at least $3$.  
However, the irregular singular point at $y=\infty$ of such a system has a Poincar\'e rank that is too high for contour integral methods to yield a general solution.  Hence while it may be possible to map \eqref{eq:general-system} to a perturbation of \eqref{eq:higher-order-model} by a smooth change of variables, the resulting simplification will not be nearly as useful as in the previously-discussed cases in general.  

\subsection{Results:  Bohr-Sommerfeld quantization rules}
One of the best-known applications of the WKBJ method and turning point theory is to address the question of existence of \emph{bound states}, by which we mean nontrivial solutions $\mathbf{w}\in L^2(\mathbb{R};\mathbb{C}^2)$ of \eqref{eq:general-system}.  Existence of such a solution requires a condition on the parameter $\lambda$, or more precisely a relationship between the parameters $\lambda$ and $\epsilon$.  Of particular interest is the characterization of the admissible values of $\lambda$, the (bound-state) \emph{eigenvalues}, in the situation that $\epsilon\ll 1$.  

For quantum-mechanical problems, a theory of bound states predating Schr\"odinger's wave mechanics was advanced by Niels Bohr and Arnold Sommerfeld; based on Bohr's correspondence principle, in which quantum systems approach classical mechanics as the Planck constant $\hbar$ tends to zero, it was hypothesized that for small $\hbar$, a classical periodic orbit describes the motion provided that its \emph{classical action} $\oint p\,\dd x$ takes one of a discrete set of values of the form $(n+\frac{1}{2})\hbar$, for $n=0,1,2,\dots$.  In other words, when $\hbar$ is small, the classical action is quantized proportional to integral multiples of $\hbar$. This (approximate) condition is called the \emph{Bohr-Sommerfeld quantization rule}.  In one dimension, solving for the momentum $p$ in terms of constant total energy $\lambda$ that is the sum of a kinetic term $p^2/(2m)$ and a potential part $\V(x)$, the classical action integral becomes $\oint\sqrt{2m(\lambda-\V(x))}\,\dd x$, where the integral is taken over a closed orbit in $x$, i.e., a classically-allowed zone traversed twice in opposite directions with opposite signs for the momentum (square root).

The Bohr-Sommerfeld quantization rule can be justified in the setting of Schr\"odinger's wave mechanics, by deriving it as an approximation directly from the Schr\"odinger equation.  The derivation can be found in many textbooks at various levels of mathematical rigor.  In this review article, we intend to apply the method of comparison equations for two turning points bounding a classically-allowed zone to give rigorous mathematical proof of Bohr-Sommerfeld quantization rules for the two special cases of the general system \eqref{eq:general-system} introduced in Section~\ref{sec:general-system}.

We first consider the Schr\"odinger equation with a potential energy function $\V$ having a single ``well'', consisting of 
a unique minimizer that we take to be at $x=0$ with minimum value $\V=0$ (both without loss of generality).  We also allow for $\V$ to be ``energy dependent'':  $\V=\V(x;\lambda)$, although for the classical case $\partial_\lambda\V(x;\lambda)=0$.
\begin{theorem}[Bohr-Sommerfeld quantization for the Schr\"odinger equation]
    \label{thm:Schroedinger-BS}
    Suppose that $\V(\diamond;\lambda)\in C^5(\mathbb{R})$, that $\V(0;\lambda)=\V'(0;\lambda)=0$ and $\V''(0;\lambda)>0$, that $\pm \V'(x;\lambda)>0$ holds for $\pm x>0$, and that for some exponents $p_\pm\ge 0$ and constants $\V_\pm>0$ we have $\V(x;\lambda)=\V_\pm |x|^{p_\pm}(1+o(1))$, $\V'(x;\lambda)=O(|x|^{p_\pm-1})$, and $\V''(x;\lambda)=O(|x|^{p_\pm-2})$ as $x\to\pm\infty$ uniformly for small $\lambda$.  Let $\lambda_\mathrm{max}>0$ and $\delta>0$ be sufficiently small, and assume that for $m=0,1,2$ and $n=0,\dots,5$, $(x;\lambda)\mapsto\partial_\lambda^{m}\partial_x^n\V(x;\lambda)$ is continuous on $(-\delta,\delta)\times [0,\lambda_\mathrm{max}]$.  Then, all eigenvalues $\lambda$ in the interval $0<\lambda\le\lambda_\mathrm{max}$ of the Schr\"odinger system (\eqref{eq:general-system} with \eqref{eq:B-Schrodinger}) satisfy the (perturbed) Bohr-Sommerfeld quantization rule
    \begin{equation}
        \cos\left(\frac{1}{\epsilon}\int_{x_-(\lambda)}^{x_+(\lambda)}\sqrt{\lambda-\V(x;\lambda)}\,\dd x\right)=O(\epsilon),\quad\epsilon\downarrow 0,
    \end{equation}
    where $x_-(\lambda)<0<x_+(\lambda)$ are the two preimages under $x\mapsto \V(x;\lambda)$ of $\lambda$.  The error term is uniform given $\lambda_\mathrm{max}$.
\end{theorem}
This implies that each eigenvalue in the indicated range lies within a distance uniformly proportional to $\epsilon^2$ of quantized approximate values $\widetilde{\lambda}_n$ for which 
\begin{equation}
\int_{x_-(\lambda)}^{x_+(\lambda)}\sqrt{\lambda-\V(x;\lambda)}\,\dd x =     \pi \epsilon \left(n+\frac{1}{2}\right), \quad \lambda=\widetilde{\lambda}_n,\quad n=0,1,2,\dots.
\end{equation}
The spacing of the quantized values themselves is proportional to $\epsilon$.  It is possible to adapt our methods to prove in addition that the eigenvalues are in $1$-to-$1$ correspondence with the approximate values $\widetilde{\lambda}_n$, but that is not the focus of this article.

Yafaev \cite{Yafaev11} gives a rigorous proof of Theorem~\ref{thm:Schroedinger-BS} under weaker hypotheses, but restricting also to the situation that the eigenvalues under consideration lie near a fixed value $\lambda>0$, which also bounds the turning points $x_\pm(\lambda)$ away from each other.  Therefore it is sufficient to treat them separately, which is very close to Langer's original approach and yields approximations of eigenfunctions near the turning points in terms of Airy functions as in Section~\ref{sec:Airy}.  See also \cite[Section 7.2.5]{Miller06}.  To analyze eigenvalues $\lambda>0$ proportional to $\epsilon$ (so-called ``low-lying eigenvalues'') the turning points must be allowed to approach each other, and for this situation rigorous analogues of Theorem~\ref{thm:Schroedinger-BS} can be found in the paper of Simon \cite{Simon83}, as well as in the very recent paper of Kristiansen and Szmolyan \cite{KristiansenS25}.  The latter work is closer to our methodology, being based on differential equations rather than perturbation theory.  Moreover, the approach of \cite{KristiansenS25} yields an analogue of Theorem~\ref{thm:Schroedinger-BS} (under different assumptions) that covers the full range of intermediate scales of $\lambda$, but the error term for the eigenvalues not in the low-lying regime is not as sharp as that given in Theorem~\ref{thm:Schroedinger-BS} (or by Yafaev \cite{Yafaev11} for the fixed-$\lambda$ regime).

In the case of the Zakharov-Shabat system (\eqref{eq:general-system} with \eqref{eq:B-ZS}), we have two potentials $A,S'$ instead of one, and it is useful to introduce the linear combinations 
\begin{equation}
    r_\pm(x):=-\frac{1}{2}S'(x)\pm A(x)
\label{eq:rpm}
\end{equation}
which have the interpretation of Riemann invariants for the hyperbolic Madelung system governing the quantities $A(x,t)^2$ and $S_x(x,t)$ in the dispersionless limit of the defocusing nonlinear Schr\"odinger equation.
In the Schr\"odinger case $\det(\mathbf{B}(x;\lambda))$ is linear in $\lambda$, but for the Zakharov-Shabat system $\det(\mathbf{B}(x;\lambda))$ is instead quadratic.  However, as a difference of squares, $-\det (\mathbf{B}(x;\lambda)) = A(x)^2 -(\lambda + \frac{1}{2}S'(x))^2 $ factors explicitly as 
\begin{equation}
-\det (\mathbf{B}(x;\lambda)) =
R_+(x;\lambda)R_-(x;\lambda),\quad R_\pm(x;\lambda) := \pm (\lambda-r_\mp(x)).
\label{eq:ZSdetB}
\end{equation}
As roots of $x\mapsto\det(\mathbf{B}(x;\lambda))$, turning points are generally associated with exactly one of the two factors $R_\pm(x;\lambda)$.   We consider the case that the two turning points are both roots of the same factor, which without loss of generality\footnote{In the Zakharov-Shabat case, the system $\epsilon\mathbf{w}'(x)=\mathbf{B}(x;\lambda)\mathbf{w}(x)$ is invariant under the substitutions $\lambda\mapsto -\lambda$, $S'\mapsto -S'$, and $\mathbf{w}\mapsto \sigma_1\mathbf{w}$, which has the effect of reflecting the diagram in Figure~\ref{fig:RiemannInvariants} through the $\lambda=0$ axis.} we take to be $R_-(x;\lambda)$. 
We further assume that $1>r_+(x)\ge\min r_+(\diamond)>\max r_-(\diamond)\ge r_-(x)>-1$.  This condition is similar to one in \cite{LiaoPlum2023} for the existence of a gap in the eigenvalue spectrum.  Its significance for us is that when $\min r_+(\diamond)<\lambda<1$, $R_+(x;\lambda)\ge \min r_+(\diamond)-\max r_-(\diamond)>0$.
\begin{figure}[h]
    \centering
    \includegraphics[width=0.5\linewidth]{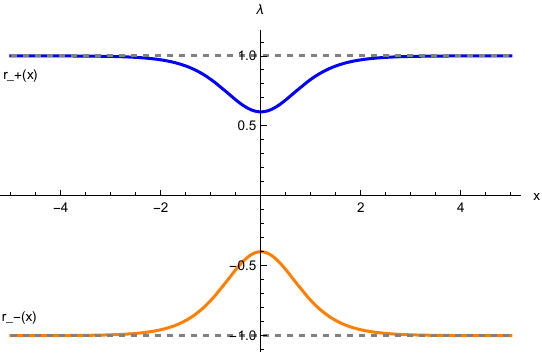}
    \caption{Illustration of $r_\pm(x)$ corresponding to initial data $A(x) = 1 - \frac{1}{2}\operatorname{sech}^2(x)$ and $S(x) = \frac{2}{10}\operatorname{tanh}(x)$. The gray dashed lines correspond to the edges of the continuous spectrum $|\lambda|\ge 1$.}
    \label{fig:RiemannInvariants}
\end{figure}

\begin{theorem}[Bohr-Sommerfeld quantization for the Zakharov-Shabat system]
Suppose that $r_\pm\in C^5(\mathbb{R})$, that $-1\le r_-(x)\le\max r_-(\diamond)<\min r_+(\diamond)\le r_+(x)\le 1$, that $r_\pm(x)\to \pm 1$ and $r_+'(x)=O(|x|^{-1})$, $r_+''(x)=O(|x|^{-2})$ as $|x|\to\infty$, and that $r_+'(0)=0$, $r_+''(0)>0$, and $\pm r_+'(x)>0$ for $\pm x>0$.  Then for each $\lambda_\mathrm{max}\in (\min r_+(\diamond),1)$, all eigenvalues $\lambda$ in the interval $\min r_+(\diamond)<\lambda\le \lambda_\mathrm{max}$ of the selfadjoint Zakharov-Shabat system (\eqref{eq:general-system} with \eqref{eq:B-ZS}) satisfy the (perturbed) Bohr-Sommerfeld quantization rule
\begin{equation}
    \cos\left(\frac{1}{\epsilon}\int_{x_-(\lambda)}^{x_+(\lambda)}\sqrt{(\lambda-r_+(x))(\lambda-r_-(x))}\,\dd x\right)=O(\epsilon),\quad\epsilon\downarrow 0,
\end{equation}
where $x_-(\lambda)<0<x_+(\lambda)$ are the two preimages under $x\mapsto r_+(x)$ of $\lambda$.  The error term is uniform given $\lambda_\mathrm{max}$.
    \label{thm:ZS-BS}
\end{theorem}
In both theorems, the integrand of the argument of the cosine has a universal form:  $\det(\mathbf{B}(x;\lambda))^{1/2}$.

The rest of the paper is devoted to the proofs of Theorems~\ref{thm:Schroedinger-BS} and \ref{thm:ZS-BS}.  As both systems are special cases of the first-order system \eqref{eq:general-system}, we start in Section~\ref{sec:mapping} by developing the method of comparison equations for such systems. We show how in the two special cases it is possible to arrive at a model system with a perturbation of the scalar potential by terms proportional to $\epsilon^2$.  We then develop the theory further in the case of a pair of turning points, defining the analogue of Langer's transformation of the independent variable first by solving separable differential equations and then obtaining expressions for the required solution that are more-or-less explicit.  Then, in Section~\ref{sec:derivatives} we prove estimates on the Langer transformation and its first three derivatives that are uniform down to the low-lying case for $\lambda$.  This leads to similarly uniform estimates of the perturbing terms of the Weber model system for two turning points.  Finally, in Section~\ref{sec:Weber-analysis} we apply these estimates to control the effect of the perturbing terms and hence express the eigenvalue condition as a correction of the Wronskian of solutions of Weber's equation decaying in opposite directions.  Appendix~\ref{sec:zeta-t-properties} describes the properties of two related analytic functions that play an important role throughout, and Appendix~\ref{sec:PC-bounds} develops bounds on quadratic expressions involving parabolic cylinder functions that are used in Section~\ref{sec:Weber-analysis}.

\section{Mapping to perturbed model systems}
\label{sec:mapping}
In this section, we will generalize a procedure \cite{Miller_Qin_2014} for reducing a first-order $2\times2$ system \eqref{eq:general-system} with traceless coefficient matrix $\mathbf{B}(x;\lambda)$ to a suitable perturbation of a model system.  In other words, we describe an analogue for first-order systems of Langer's method \cite{Langer34,Langer35,Langer37}.

Starting from \eqref{eq:general-system} with $\mathrm{tr}(\mathbf{B}(x;\lambda))=0$, 
the first step is to introduce a gauge transformation with unit-determinant matrix
$\mathbf{G}(x;\lambda)$ by setting:
\begin{equation}
    \mathbf{w}(x;\lambda) = \mathbf{G}(x;\lambda) \mathbf{s}(x;\lambda).
\end{equation}
The equation satisfied by the new unknown $\mathbf{s}(x;\lambda)$ is given by
\begin{equation}\label{eq: gauge transformed equation}
    \epsilon \frac{\dd\mathbf{s}}{\dd x}(x;\lambda) = \mathbf{G}^{-1}(x;\lambda)\mathbf{B}(x;\lambda)\mathbf{G}(x;\lambda)\mathbf{s}(x;\lambda) - \epsilon \mathbf{G}^{-1}(x;\lambda)\mathbf{G}'(x;\lambda)\mathbf{s}(x;\lambda).
\end{equation}
Notice that this is a perturbation of a version of the original system
in which the coefficient matrix $\mathbf{B}(x;\lambda)$ has undergone conjugation with $\mathbf{G}(x;\lambda)$. Since the determinant is invariant under conjugation, in order to arrive at a leading-order coefficient matrix of a desired target form, it is necessary to introduce an additional scalar factor. This is achieved via a
\textit{Langer transformation}, a change of the independent variable $x$ to $y=g(x;\lambda)$. Indeed, let $x\mapsto g(x;\lambda)$ be a smooth and monotone map (so that the inverse map $x=g^{-1}(y;\lambda)$ exists), and define 
\begin{equation*}
    \mathbf{t}(y;\lambda) :=  \mathbf{s}(x=g^{-1}(y;\lambda);\lambda)
\end{equation*}
By the chain rule one obtains the equivalent equation for $\mathbf{t}(y;\lambda)$:
\begin{equation}\label{eq: gauge transformed equation with Langer}
        \epsilon \frac{\dd \mathbf{t}}{\dd y}(y;\lambda) = \frac{1}{g'(x;\lambda)}\Big( \mathbf{G}(x;\lambda)^{-1} 
        \mathbf{B}(x;\lambda)  \mathbf{G}(x;\lambda) \Big) \mathbf{t}(y;\lambda) 
        - \epsilon \frac{1}{g'(x;\lambda)} \mathbf{G}(x;\lambda)^{-1} \mathbf{G}'(x;\lambda)\mathbf{t}(y;\lambda).
    \end{equation}
Here we think of the coefficient matrices on the right-hand side as functions of $y$ via $x=g^{-1}(y;\lambda)$. \subsection{Determining the form of the gauge transformation $\mathbf{G}(x;\lambda)$}
Suppose we require that the leading-order coefficient matrix on the right-hand side of \eqref{eq: gauge transformed equation with Langer} have the form 
\begin{equation}
\mathbf{M}(y)=\begin{bmatrix}
    0 & 1 \\
    m(y) & 0 
\end{bmatrix},
\end{equation}
so that after neglecting small coefficients one has a system equivalent to a second-order scalar equation $\epsilon^2 t_1''(y;\lambda)=m(y)t_1(y;\lambda)$.
Thus we require:
\begin{equation}\label{eq: modelM}
    \frac{1}{g'(x;\lambda)}\Big( \mathbf{G}(x;\lambda)^{-1} 
        \mathbf{B}(x;\lambda)  \mathbf{G}(x;\lambda) \Big)
        = \mathbf{M}(y).
\end{equation}
From this, we take determinants and immediately derive the relationship
\begin{equation}\label{eq:detB-detM-relationship}
    -\frac{1}{g'(x;\lambda)} \det(\mathbf{B}(x;\lambda)) = -g'(x;\lambda)\det(\mathbf{M}(y)) = g'(x;\lambda)m(y),
\end{equation}
which may be regarded as a first-order differential equation for $y=g(x;\lambda)$.  The existence of a suitable solution defining an invertible Langer transformation is not trivial and will be considered later.  However we may note that if $\det(\mathbf{B}(x;\lambda))$ has a fixed sign and one takes $m(y)=-\mathrm{sgn}(\det(\mathbf{B}(x;\lambda)))$ independent of $y$, then $g'(x;\lambda)=\sqrt{|\det(\mathbf{B}(x;\lambda))|}$, which upon integration produces the well-known Liouville-Green transformation $y=g(x;\lambda)$ in terms of which solutions of the model problem for $m(y)= -1$ (resp., $m(y)=1$) are as shown in \eqref{eq:model-solution-minus-one} (resp., \eqref{eq:model-solution-plus-one}).  These reproduce the standard WKBJ approximation formul\ae\ for problems without turning points. 

Using \eqref{eq:detB-detM-relationship}, we can rewrite \eqref{eq: modelM} as a homogeneous linear system for the elements of $\mathbf{G}(x;\lambda)$:
\begin{equation}\label{eq: systemforG}
\begin{bmatrix}
    g'(x;\lambda)B_{11}(x;\lambda) & \det(\mathbf{B}(x;\lambda)) & g'(x;\lambda)B_{12}(x;\lambda) & 0 \\
    -g'(x;\lambda) & B_{11}(x;\lambda) & 0 & B_{12}(x;\lambda) \\
    g'(x;\lambda)B_{21}(x;\lambda) & 0 & g'(x;\lambda)B_{22}(x;\lambda) & \det(\mathbf{B}(x;\lambda)) \\
    0 & B_{21}(x;\lambda) & -g'(x;\lambda) & B_{22}(x;\lambda)
\end{bmatrix}
\begin{bmatrix}
     G_{11}(x;\lambda) \\ G_{12}(x;\lambda) \\
     G_{21}(x;\lambda) \\ G_{22}(x;\lambda) \\
\end{bmatrix} = \mathbf{0}
\end{equation}
Under the assumption in force that $B_{11}(x;\lambda)+B_{22}(x;\lambda)=0$, there exist nontrivial solutions of this system, because the determinant of the matrix on the left-hand side is $g'(x;\lambda)^2\det(\mathbf{B}(x;\lambda))\mathrm{tr}(\mathbf{B}(x;\lambda))^2=0$.   
Moreover, row reduction shows that the homogeneous system \eqref{eq: systemforG} has rank 2, and thus $\mathbf{G}(x;\lambda)$ can be written in the form 
\begin{equation}\label{Gstructure}
\mathbf{G}(x;\lambda)=
\begin{bmatrix}
    B_{11}(x;\lambda)n_1(x;\lambda) + B_{12}(x;\lambda)n_2(x;\lambda) & g'(x;\lambda)n_1(x;\lambda) \\
    B_{21}(x;\lambda)n_1(x;\lambda) + B_{22}(x;\lambda)n_2(x;\lambda) & g'(x;\lambda)n_2(x;\lambda)
\end{bmatrix}
\end{equation}
where $n_1(x;\lambda)$ and $n_2(x;\lambda)$ are coordinates for the nullspace that are free up to an overall scalar multiple determined so that $\det(\mathbf{G}(x;\lambda))=1$ holds. With this form of $\mathbf{G}(x;\lambda)$, equation \eqref{eq: gauge transformed equation with Langer} can now be written in the form
\begin{equation}\label{eq:perturbed-model}
        \epsilon \frac{\dd  \mathbf{t}}{\dd y}(y;\lambda)= \mathbf{M}(y) \mathbf{t}(y;\lambda) \\
        - \epsilon \mathbf{P}(x=g^{-1}(y;\lambda);\lambda) \mathbf{t}(y;\lambda)
        ,\hspace{1em}
        \mathbf{P}(x;\lambda) := \frac{1}{g'(x;\lambda)} \mathbf{G}^{-1}(x;\lambda)\mathbf{G}'(x;\lambda).
    \end{equation}
Using \eqref{Gstructure} and $\det(\mathbf{G}(x;\lambda))=1$, the elements of $\mathbf{P}(x;\lambda)$ can be calculated as follows:
\begin{equation}
    \begin{split}
        P_{11}(x;\lambda)=-P_{22}(x;\lambda)&=\frac{1}{2}g'(x;\lambda)n_2(x;\lambda)^2\left[\frac{\dd}{\dd x}\left(\frac{B_{11}(x;\lambda)\rho(x;\lambda) + B_{12}(x;\lambda)}{g'(x;\lambda)}\right)\right.\\
        &\quad\left.-\rho(x;\lambda)^2\frac{\dd}{\dd x}\left(\frac{B_{21}(x;\lambda)\rho(x;\lambda)-B_{11}(x;\lambda)}{g'(x;\lambda)\rho(x;\lambda)}\right)\right]\\
        P_{12}(x;\lambda)&=g'(x;\lambda)n_2(x;\lambda)^2\rho'(x;\lambda)=-\frac{\det(\mathbf{B}(x;\lambda))}{g'(x;\lambda)m(z)}n_2(x;\lambda)^2\rho'(x;\lambda)\\
        P_{21}(x;\lambda)&=\frac{n_2(x;\lambda)^2}{g'(x;\lambda)}(B_{11}(x;\lambda)\rho(x;\lambda)+B_{12}(x;\lambda))^2\frac{\dd}{\dd x}\left(\frac{B_{21}(x;\lambda)\rho(x;\lambda)-B_{11}(x;\lambda)}{B_{11}(x;\lambda)\rho(x;\lambda)+B_{12}(x;\lambda)}\right),
    \end{split}
    \label{eq:P-elements}
\end{equation}
where $\rho(x;\lambda):=n_1(x;\lambda)/n_2(x;\lambda)$, and the second expression for $P_{12}(x;\lambda)$ follows from the first using \eqref{eq:detB-detM-relationship}.
\subsection{Pushing the perturbation to higher order}
Because the factor of $\epsilon$ is paired with both with the differential term on the left-hand side of equation \eqref{eq:perturbed-model} and the perturbation $\mathbf{P}(x;\lambda)$, the perturbation term is not directly controllable. However, the terms proportional to $\epsilon$ on the right-hand side can be removed with a near-identity transformation:
\begin{equation}\label{eq: nilpotent perturbation}
        \mathbf{t}(y;\lambda) = (\mathbb{I}+\epsilon \mathbf{H}(x;\lambda) )\mathbf{u}(y;\lambda)
    \end{equation}
where $\mathbf{H}(x;\lambda)$ is independent of $\epsilon$. It will be useful if the transformation can be inverted explicitly, which is the case when $\mathbf{H}(x;\lambda)$ is a nilpotent matrix. Such a matrix can be expressed as an outer product in terms of a vector $\mathbf{h}(x;\lambda): = (h_1(x;\lambda),h_2(x;\lambda))^\top$:
\begin{equation}\label{def:nilpotent-H}
        \mathbf{H}(x;\lambda) = \ii\sigma_2\mathbf{h}(x;\lambda)\mathbf{h}(x;\lambda)^\top 
        = \begin{bmatrix}
            h_1(x;\lambda)h_2(x;\lambda) & h_2(x;\lambda)^2 \\
            -h_1(x;\lambda)^2 & -h_1(x;\lambda)h_2(x;\lambda)
        \end{bmatrix}
    \end{equation}
Since $\mathbf{H}(x;\lambda)^2 = \mathbf{0}$, $(\mathbb{I}+\epsilon \mathbf{H}(x;\lambda))^{-1} = \mathbb{I}-\epsilon \mathbf{H}(x;\lambda) $ and we may rewrite equation \eqref{eq: gauge transformed equation with Langer} in terms of the new unknown $\mathbf{u}(y;\lambda)$. This equation is of the form
\begin{equation}\label{eq:u}
        \epsilon \frac{\dd \mathbf{u}}{\dd y}(y;\lambda) = \Big( \mathbf{M }(y) + \epsilon \mathbf{C}_1(x;\lambda) + \epsilon^2 \mathbf{C}_2(x;\lambda) + \epsilon^3 \mathbf{C}_3(x;\lambda) \Big)\mathbf{u}(y;\lambda)
    \end{equation}
    \begin{align*}
        \mathbf{C}_1(x;\lambda) &:= [\mathbf{M}(y),\mathbf{H}(x;\lambda)] - \mathbf{P}(x;\lambda) \\
        \mathbf{C}_2(x;\lambda) &:= -[\mathbf{P}(x;\lambda),\mathbf{H}(x;\lambda)] - \mathbf{H}(x;\lambda)\mathbf{M}(y)\mathbf{H}(x;\lambda) - \frac{1}{g'(x;\lambda)} \mathbf{H}'(x;\lambda) \\
        \mathbf{C}_3(x;\lambda) &:= \frac{1}{g'(x;\lambda)} \mathbf{H}(x;\lambda) \mathbf{H}'(x;\lambda) + \mathbf{H}(x;\lambda)\mathbf{P}(x;\lambda)\mathbf{H}(x;\lambda)
    \end{align*}
The purpose of introducing the transformation from $\mathbf{t}(y;\lambda)$ to $\mathbf{u}(y;\lambda)$ was to remove terms proportional to $\epsilon$ from the right-hand side, so we demand that $\mathbf{C}_1(x;\lambda)=\mathbf{0}$. Note that $[\mathbf{M}(y),\mathbf{H}(x;\lambda)]$ is a matrix of the form
\begin{equation}
    [\mathbf{M}(y),\mathbf{H}(x;\lambda)]=\begin{bmatrix}-h_1(x;\lambda)^2-m(y)h_2(x;\lambda)^2 & -2h_1(x;\lambda)h_2(x;\lambda)\\2m(y)h_1(x;\lambda)h_2(x;\lambda) & h_1(x;\lambda)^2+m(y)h_2(x;\lambda)^2\end{bmatrix}.
\end{equation}
In particular, $\mathbf{C}_1(x;\lambda)=\mathbf{0}$ requires that $\mathrm{tr}(\mathbf{P}(x;\lambda))=0$, which is automatically satisfied according to \eqref{eq:P-elements}.  A second condition this imposes on $\mathbf{P}(x;\lambda)$ without consideration of the unknown vector $\mathbf{h}(x;\lambda)$ is that
\begin{equation}\label{Pstructure}
P_{21}(x;\lambda) = -m(y)P_{12}(x;\lambda).
\end{equation}
Again using \eqref{eq:P-elements}, the condition \eqref{Pstructure} is satisfied if the ratio $\rho(x;\lambda)=n_1(x;\lambda)/n_2(x;\lambda)$ satisfies a Riccati equation: 
\begin{multline}\label{eq:Riccati}
    \rho'(x;\lambda) = \Bigg(\frac{B_{11}(x;\lambda)B_{21}'(x;\lambda) - B_{21}(x;\lambda)B_{11}'(x;\lambda)}{\det \mathbf{B}(x;\lambda)}\Bigg)\rho(x;\lambda)^2\\ + \Bigg( \frac{B_{12}(x;\lambda)B_{21}'(x;\lambda)-B_{21}(x;\lambda)B_{12}'(x;\lambda)}{\det \mathbf{B}(x;\lambda)}\Bigg)\rho(x;\lambda)\\ + \Bigg( \frac{B_{11}(x;\lambda)B_{12}'(x;\lambda)-B_{12}(x;\lambda)B_{11}'(x;\lambda)}{\det \mathbf{B}(x;\lambda)}\Bigg).
\end{multline}
From a solution of this equation we may eliminate $n_1(x;\lambda)$ via  $n_1(x;\lambda)=n_2(x;\lambda)\rho(x;\lambda)$, and then using \eqref{Gstructure} we set $\det(\mathbf{G}(x;\lambda))=1$ and find $n_2(x;\lambda)^2$ in the form  
\begin{equation}\label{p2-constantrho}
    n_2(x;\lambda)^2 = \Bigg( g'(x;\lambda)\Big(B_{12}(x;\lambda) + 2B_{11}(x;\lambda)\rho(x;\lambda) - B_{21}(x;\lambda)\rho(x;\lambda)^2\Big)\Bigg)^{-1}.
\end{equation}
With $n_2(x;\lambda)^2$ determined, there remain two additional conditions in the equation $\mathbf{C}_1(x;\lambda)=\mathbf{0}$, which may be viewed as conditions on the elements of $\mathbf{h}(x;\lambda)$: 
\begin{equation}
    h_1(x;\lambda)h_2(x;\lambda)=-\frac{1}{2}P_{12}(x;\lambda)
    \label{eq:product-of-h1h2}
\end{equation}
and
\begin{equation}h_1(x;\lambda)^2+m(y)h_2(x;\lambda)^2=P_{22}(x;\lambda)
\label{eq:squares-of-h1h2}
\end{equation}
where in both conditions, the right-hand side is determined by \eqref{eq:P-elements} with \eqref{p2-constantrho}.

\subsection{The case of constant $\rho(x;\lambda)$}
In the two examples that we will explore in this article, there exist constant solutions $\rho(x;\lambda)$ of the Riccati equation \eqref{eq:Riccati}. If $\rho(x;\lambda)$ is constant, then according to \eqref{eq:P-elements}, we automatically obtain $P_{12}(x;\lambda) = 0$. Then from \eqref{Pstructure}, $P_{21}(x;\lambda) = 0$ also, so $\mathbf{P}(x;\lambda) = -P_{22}(x;\lambda)\sigma_3$ (diagonal matrix).  Furthermore, once it is known that $P_{12}(x,\lambda)=0$, from \eqref{eq:product-of-h1h2} we see that either $h_1(x;\lambda)=0$ or $h_2(x;\lambda)=0$ (for each $x\in\mathbb{R}$, however we shall assume that given $\lambda$ either $h_1$ or $h_2$ vanishes identically).

To ensure that the perturbation in $\mathbf{C}_2$ amounts to a scalar perturbation of the model potential $m(y)$ it is necessary to assume that $h_2(x;\lambda)\equiv 0$.  Then \eqref{eq:squares-of-h1h2} gives $h_1(x;\lambda)^2=P_{22}(x;\lambda)$, and a computation shows that
\begin{equation}    \mathbf{C}_2(x;\lambda)=\begin{bmatrix}0&0\\Q(y;\lambda) & 0\end{bmatrix}
\end{equation}
in which
\begin{equation}
    Q(y;\lambda):=2h_1(x;\lambda)^2P_{22}(x;\lambda)-h_1(x;\lambda)^4+\frac{1}{g'(x;\lambda)}\frac{\dd}{\dd x}h_1(x;\lambda)^2=P_{22}(x;\lambda)^2+\frac{P_{22}'(x;\lambda)}{g'(x;\lambda)},
    \label{eq:Q-constant-rho}
\end{equation}
where we substitute $x=g^{-1}(y;\lambda)$ to define $Q(y;\lambda)$,
and we also obtain that $\mathbf{C}_3(x;\lambda)=\mathbf{0}$.  The perturbing terms in the differential equation \eqref{eq:u} then amount to a correction of the coefficient $m(y)$ by $\epsilon^2 Q(y;\lambda)$.

We now explain the details for the two applications of interest in this paper.
\subsubsection{The Schr\"{o}dinger equation}
Recall the coefficient matrix $\mathbf{B}(x;\lambda)$ from \eqref{eq:B-Schrodinger}.  
Then as $B_{11}(x;\lambda)=B_{22}(x;\lambda)=0$ while $B_{12}(x;\lambda)=1$ and $B_{21}(x;\lambda)=\V(x;\lambda)-\lambda$, the Riccati equation \eqref{eq:Riccati} reduces to the linear equation
\begin{equation}
    \rho'(x;\lambda) = -\frac{\V'(x;\lambda)}{\V(x;\lambda)-\lambda}\rho(x;\lambda)
\end{equation} 
which admits the constant (trivial) solution $\rho(x;\lambda)\equiv 0$.  Thus also $n_1(x;\lambda)=0$ and from \eqref{p2-constantrho} we obtain $n_2(x;\lambda)=g'(x;\lambda)^{-1/2}$ after taking a positive square root.
According to \eqref{Gstructure}, the gauge matrix $\mathbf{G}(x;\lambda)$ in this case takes the simple diagonal form $\mathbf{G}(x;\lambda)=g(x;\lambda)^{-\sigma_3/2}$. From \eqref{eq:P-elements} we then obtain that
\begin{equation}
    P_{22}(x;\lambda)=-\frac{1}{2}\frac{\dd}{\dd x}\frac{1}{g'(x;\lambda)}=\frac{g''(x;\lambda)}{2g'(x;\lambda)^2}.
\end{equation}
Thus, the nilpotent matrix $\mathbf{H}(x;\lambda)$ with $h_2(x;\lambda)=0$ is simply
\begin{equation}
    \mathbf{H}(x;\lambda)=\begin{bmatrix}0 & 0\\-P_{22}(x;\lambda) & 0\end{bmatrix} = \begin{bmatrix}0 & 0\\-g''(x;\lambda)/(2g'(x;\lambda)^2) & 0\end{bmatrix}.
\end{equation}
Using \eqref{eq:Q-constant-rho} then gives
\begin{equation}
    Q(y;\lambda)=\frac{g''(x;\lambda)^2}{4g'(x;\lambda)^4} +\frac{1}{g'(x;\lambda)}\left[\frac{g'''(x;\lambda)}{2g'(x;\lambda)^2} -\frac{g''(x;\lambda)^2}{g'(x;\lambda)^3}\right]=-\frac{3g''(x;\lambda)^2}{4g'(x;\lambda)^4} +\frac{g'''(x;\lambda)}{2g'(x;\lambda)^3},\quad y=g(x;\lambda).
    \label{eq:Schroedinger-Q}
\end{equation}
Combining the results shows that 
the total gauge transformation $\mathbf{w}=\mathbf{G}(\mathbb{I}+\epsilon\mathbf{H})\mathbf{u}$ from the original unknown $\mathbf{w}(x;\lambda)$ to $\mathbf{u}(y;\lambda)$ can be written in the form
\begin{equation}
\begin{split}
    \mathbf{G}(x;\lambda)(\mathbb{I}+\epsilon\mathbf{H}(x;\lambda)) &=\begin{bmatrix}
        a(x;\lambda) & 0 \\
        0 & a(x;\lambda)g'(x;\lambda)
    \end{bmatrix} \Bigg( \mathbb{I} + \epsilon\begin{bmatrix}
        0 & 0 \\
       a'(x;\lambda)/(a(x;\lambda)g'(x;\lambda)) & 0
    \end{bmatrix}\Bigg)\\ &= 
    \begin{bmatrix}
    a(x;\lambda) & 0 \\
    \epsilon a'(x;\lambda) & a(x;\lambda)g'(x;\lambda)
    \end{bmatrix}
\end{split}
\end{equation}
in which $a(x;\lambda):=g'(x;\lambda)^{-\frac{1}{2}}$.

\subsubsection{The Zakharov-Shabat system} 
Recall the matrix $\mathbf{B}(x;\lambda)$ defined in \eqref{eq:B-ZS} relevant to the Zakharov-Shabat system. In this case, the Riccati equation \eqref{eq:Riccati} reduces to:
\begin{equation}
\rho'(x;\lambda) = \frac{-\ii\Big(\lambda + \frac{1}{2}S'(x)\Big)A'(x) + \frac{1}{2}\ii S''(x)A(x)}{\Big(\lambda + \frac{1}{2}S'(x)\Big)^2 - A(x)^2}\Bigg( \rho(x;\lambda)^2+1\Bigg)
\end{equation} 
which admits constant solutions $\rho(x;\lambda) = \pm \ii$. So also $n_1(x;\lambda)=\pm\ii n_2(x;\lambda)$.  With $\rho=\pm\ii$, \eqref{p2-constantrho} becomes
\begin{equation}
    n_2(x;\lambda)^2 = 
        \Bigg( 2g'(x;\lambda) R_\pm(x;\lambda)\Bigg)^{-1}. 
\end{equation}
Under our assumptions, given $\lambda$, we choose the sign  of $\rho=\pm\ii$ so that $R_\pm(x;\lambda)$ is nonvanishing, and therefore $n_2(x;\lambda)^2$ is bounded. Handling both cases simultaneously, the possible gauge transformations $\mathbf{G}(x;\lambda)$ take the form
\begin{equation}
\begin{split}
    \mathbf{G}(x;\lambda) &= n_2(x;\lambda)
    \begin{bmatrix}
    \rho(x;\lambda) B_{11}(x;\lambda) + B_{12}(x;\lambda) & \rho(x;\lambda) g'(x;\lambda) \\
    \rho(x;\lambda) B_{21}(x;\lambda) -B_{11}(x;\lambda) & g'(x;\lambda)
    \end{bmatrix} \\
    &= \Bigg( 2g'(x;\lambda) R_\pm(x;\lambda)\Bigg)^{-1/2}
    \begin{bmatrix}
    R_\pm(x;\lambda) & \pm \ii g'(x;\lambda)\\
    \pm \ii R_\pm(x;\lambda) & g'(x;\lambda)
    \end{bmatrix}.
    \end{split}
\end{equation}
From \eqref{eq:P-elements}, we have
\begin{equation*}
    P_{22}(x;\lambda) = \frac{1}{2g'(x;\lambda)}\frac{\dd}{\dd x}\left( \log\left( \frac{g'(x;\lambda)}{R_\pm(x;\lambda)}\right)\right),
\end{equation*}
which, by \eqref{eq:Q-constant-rho}, leads to
\begin{equation}
    Q(y;\lambda) = -\frac{3g''(x;\lambda)^2}{4g'(x;\lambda)^4} + \frac{g'''(x;\lambda)}{2g'(x;\lambda)^3}-\frac{1}{g'(x;\lambda)^2}\left[-\frac{3R_\pm'(x;\lambda)^2}{4R_\pm(x;\lambda)^2} + \frac{R_\pm''(x;\lambda)}{2R_\pm(x;\lambda)}\right],\quad y=g(x;\lambda).
    \label{eq:ZS-Q}
\end{equation}

\subsection{Weber model and Langer transformations for problems with two turning points}
Now we focus on the situation germane to Theorems~\ref{thm:Schroedinger-BS} and \ref{thm:ZS-BS}: that $\lambda\in\mathbb{R}$ is such that there exist two simple real roots (turning points) $x=x_\pm(\lambda)$ of $\det(\mathbf{B}(x;\lambda))$ and such that $\det(\mathbf{B}(x;\lambda))<0$ holds for $x<x_-(\lambda)$ and for $x>x_+(\lambda)$, while $\det(\mathbf{B}(x;\lambda))>0$ holds for $x_-(\lambda)<x<x_+(\lambda)$.  To model this situation, we assume that the model potential $m(y)$ is a quadratic function of $y$ with a parameter $b<0$:
\begin{equation}
    m(y)=m(y;b):=\frac{1}{4}y^2+b,\quad b<0.
\end{equation}

The Langer transformation $x\mapsto y=g(x;\lambda)$ is necessarily a solution of the first-order ordinary differential equation \eqref{eq:detB-detM-relationship}, which now takes the form
\begin{equation}
    g'(x;\lambda)^2\cdot\left(\frac{1}{4}g(x;\lambda)^2+b\right)=-\det(\mathbf{B}(x;\lambda)).
    \label{eq:Langer-ODE}
\end{equation}
Since $y=g(x;\lambda)$ is to be invertible, we want $g'(x;\lambda)>0$ for all $x\in\mathbb{R}$, so it is necessary that $g(x;\lambda)$ satisfy two auxiliary conditions:
\begin{equation}
    \frac{1}{4}g(x_\pm(\lambda);\lambda)^2 + b = 0 \iff g(x_\pm(\lambda);\lambda))=\pm 2\sqrt{-b}.
    \label{eq:LangerICs}
\end{equation}
Imposing only the condition at, say, $x=x_+(\lambda)$, we may separate the variables and integrate from this turning point to the right:
\begin{equation}
    \int_{2\sqrt{-b}}^y\sqrt{\frac{1}{4}\bar{y}^2+b}\,\dd \bar{y} = \int_{x_+(\lambda)}^x\sqrt{-\det(\mathbf{B}(\bar{x};\lambda))}\,\dd \bar{x},\quad x>x_+(\lambda),
    \label{eq:LangerImplicitRight}
\end{equation}
and to the left:
\begin{equation}
    \int_{2\sqrt{-b}}^y\sqrt{-b-\frac{1}{4}\bar{y}^2}\,\dd \bar{y} = \int_{x_+(\lambda)}^x\sqrt{\det(\mathbf{B}(\bar{x};\lambda))}\,\dd \bar{x},\quad x_-(\lambda)<x<x_+(\lambda).
\label{eq:LangerImplicitMid}
\end{equation}
Letting $x\downarrow x_-(\lambda)$ in this formula and enforcing the remaining condition from \eqref{eq:LangerICs} on $y=g(x;\lambda)$ at $x=x_-(\lambda)$ yields a condition determining $b<0$ in terms of $\lambda\in\mathbb{R}$:
\begin{equation}
\begin{split}
   -\frac{1}{\pi} \int_{x_-(\lambda)}^{x_+(\lambda)}\sqrt{\det(\mathbf{B}(\bar{x};\lambda))}\,\dd \bar{x}&=-\frac{1}{\pi}\int_{-2\sqrt{-b}}^{2\sqrt{-b}}\sqrt{-b-\frac{1}{4}\bar{y}^2}\,\dd \bar{y}\\ & = 
    \frac{2b}{\pi}\int_{-1}^1\sqrt{1-\gamma^2}\,\dd\gamma\\
    &=b.
\end{split}
\label{eq:b-of-lambda}
\end{equation}
Then, integrating to the left from $x=x_-(\lambda)$ imposing that $g(x_-(\lambda);\lambda)=-2\sqrt{-b}$ gives
\begin{equation}
    \int_{-2\sqrt{-b}}^y\sqrt{\frac{1}{4}\bar{y}^2+b}\,\dd \bar{y} = \int_{x_-(\lambda)}^x\sqrt{-\det(\mathbf{B}(\bar{x};\lambda))}\,\dd \bar{x},\quad x<x_-(\lambda),
\label{eq:LangerImplicitLeft}
\end{equation}
and when $b=b(\lambda)$ satisfies \eqref{eq:b-of-lambda}, the three equations \eqref{eq:LangerImplicitRight}, \eqref{eq:LangerImplicitMid}, and \eqref{eq:LangerImplicitLeft} together define $y=g(x;\lambda)$ as a continuous function of $x\in\mathbb{R}$ that satisfies the differential equation \eqref{eq:Langer-ODE} except possibly at the turning points $x=x_\pm(\lambda)$ and that satisfies both auxiliary conditions \eqref{eq:LangerICs}.

In the case of the Schr\"odinger equation, where $\mathbf{B}(x;\lambda)$ is given by \eqref{eq:B-Schrodinger}, we have simply $\det(\mathbf{B}(x;\lambda))=\lambda-\V(x;\lambda)$.  
We now show that in the Zakharov-Shabat case, where $\mathbf{B}(x;\lambda)$ is instead given by \eqref{eq:B-ZS},  $\det(\mathbf{B}(x;\lambda))$ can be written in a similar form at the cost of a change of coordinate $x$.  Indeed, according to the assumptions of Theorem~\ref{thm:ZS-BS}, there is a factor $\lambda-r_-(x)$ of $\det(\mathbf{B}(x;\lambda))$ that is strictly positive and of class $C^5(\mathbb{R})$, so introducing a shifted spectral parameter $\widetilde{\lambda}:=\lambda-\min r_+(\diamond)$ and variable $\widetilde{x}$ via the invertible transformation
\begin{equation}
    \widetilde{x}=\widetilde{x}(x;\lambda):=\int_0^x\sqrt{\lambda-r_-(\bar{x})}\,\dd \bar{x}\implies  \frac{\dd\widetilde{x}}{\dd x}(x;\lambda)=\sqrt{\lambda-r_-(x)},
\end{equation}
and then defining an energy-dependent potential function by \begin{equation}
    \V(\widetilde{x};\widetilde{\lambda}):=r_+(x)-\min r_+(\diamond),\quad \widetilde{x}=\widetilde{x}(x;\lambda),
    \label{eq:ZS-potential}
\end{equation} 
(i.e., composition of $r_+(x)$, shifted to make the minimum value zero, with the inverse  $x=x(\widetilde{x};\lambda)$) and setting $\widetilde{x}_\pm(\widetilde{\lambda}):=\widetilde{x}(x_\pm(\lambda);\lambda)$ yields the differential identity $\sqrt{\pm\det(\mathbf{B}(x;\lambda))}\,\dd x = \sqrt{\pm(\lambda-\V(\widetilde{x};\widetilde{\lambda}))}\,\dd \widetilde{x}$.  Hence the integrals appearing on the right-hand sides of \eqref{eq:LangerImplicitRight}, \eqref{eq:LangerImplicitMid}, and \eqref{eq:LangerImplicitLeft} can be written in a common form for both the Schr\"odinger and Zakharov-Shabat cases; after dropping the tilde, in both cases the Langer transformation is implicitly defined for $x>x_-(\lambda)$ by the equation
\begin{equation}
    \int_{2\sqrt{-b}}^y\sqrt{\left|\frac{1}{4}\bar{y}^2+b\right|}\,\dd \bar{y} = \int_{x_+(\lambda)}^x\sqrt{|\V(\bar{x};\lambda)-\lambda|}\,\dd \bar{x},\quad x>x_-(\lambda),
\label{eq:general-Langer}
\end{equation}
where
\begin{equation}
    b=-\frac{1}{\pi}\int_{x_-(\lambda)}^{x_+(\lambda)}\sqrt{\lambda-\V(\bar{x};\lambda)}\,\dd \bar{x}.
    \label{eq:general-b}
\end{equation}
We then have the following, whose simple proof is omitted.
\begin{lemma}[Universal form for Langer transformation]
    The function $\V(x;\lambda)$ defined from $r_\pm(\diamond)$ satisfying the hypotheses of Theorem~\ref{thm:ZS-BS} by \eqref{eq:ZS-potential} satisfies the hypotheses of Theorem~\ref{thm:Schroedinger-BS} with exponents $p_\pm=0$ and leading coefficients $\V_\pm=1-\min r_+(\diamond)$.
    \label{lem:ZS-Schroedinger}
\end{lemma}

Therefore, to analyze the Langer transformation in both cases, it suffices to analyze $y=g(x;\lambda)$ as defined by \eqref{eq:general-Langer} and \eqref{eq:general-b} assuming the hypotheses of Theorem~\ref{thm:Schroedinger-BS}.  Going further, we can express $y=g(x;\lambda)$ explicitly in terms of the analytic function $\zeta:(-1,+\infty)\to (\zeta(-1),+\infty)$ and its analytic inverse $t:(\zeta(-1),+\infty)\to (-1,+\infty)$, both defined in Appendix~\ref{sec:zeta-t-properties}.  Indeed, for $x>x_+(\lambda)$, by the substitution $\tau=\bar{y}/(2\sqrt{-b})$, the integral on the left-hand side of \eqref{eq:general-Langer} becomes simply $-\frac{4}{3}b\zeta(y/(2\sqrt{-b}))^{3/2}$.  Inverting this using the inverse function $t(\diamond)$ to $\zeta(\diamond)$, we obtain an explicit formula for $y=g(x;\lambda)$; if $x>x_+(\lambda)$ this formula reads
\begin{equation}
    y=g(x;\lambda)=2\sqrt{-b}t\left(\left[-\frac{3}{4b}J(x;\lambda)\right]^{2/3}\right),\quad J(x;\lambda):=\int_{x_+(\lambda)}^x\sqrt{\V(\bar{x};\lambda)-\lambda}\,\dd\bar{x},\quad x>x_+(\lambda).
    \label{eq:g-x-large}
\end{equation}
This formula is especially useful for the analysis of $g(x;\lambda)$ for $x$ bounded away from the turning points converging to $0$ as $\lambda\downarrow 0$.

On the other hand, a better formula for the complementary case that $x$ is in a neighborhood of the turning points is obtained by writing the right-hand side of \eqref{eq:general-Langer} in a similar form as follows.  Firstly, according to the hypotheses of Theorem~\ref{thm:Schroedinger-BS} we may define a function $s=s(x)$ by
\begin{equation}
    s(x;\lambda):=\mathrm{sgn}(x)\sqrt{\V(x;\lambda)},\quad x\in\mathbb{R}.
\label{eq:s-of-x}
\end{equation}
\begin{lemma}[Smoothness of $x\mapsto s(x;\lambda)$ and its inverse]

Suppose that $\V$ satisfies the hypotheses of Theorem~\ref{thm:Schroedinger-BS}.
Then, the function $s(\diamond;\lambda)$ is of class $C^3(\mathbb{R})$, and for $m=0,1,2$ and $n=0,1,2,3$, $(x;\lambda)\mapsto \partial_\lambda^m\partial_x^n s(x;\lambda)$ is continuous on $(-\delta,\delta)\times [0,\lambda_\mathrm{max}]$.  Moreover, $s(0;\lambda)=0$ and $s(x;\lambda)$ is strictly increasing, hence there is an inverse function denoted $x=x(s;\lambda)$, which is also of class $C^{3}$ on its domain with $\partial_\lambda^m\partial_s^n x(s;\lambda)$ continuous on $(-\nu,\nu)\times [0,\lambda_\mathrm{max}]$ for some small $\nu>0$ and $m=0,1,2$, $n=0,1,2,3$.
\label{lem:x}
    \end{lemma}
    \begin{proof}
        Since $\V(x;\lambda)$ only vanishes at $x=0$, it is clear that $s(\diamond;\lambda)$ has $5$ continuous derivatives except possibly at $x=0$.  To analyze $s$ near $x=0$, we first write $\V$ in the form
        \begin{equation}
            \V(x;\lambda)=x^2j(x;\lambda),\quad 
            j(x;\lambda):=\int_0^1\int_0^1t\V''(tux;\lambda)\,\dd t\, \dd u,
            \label{eq:V-xsquared-j}
        \end{equation}
        where we see $j(\diamond;\lambda)\in C^3(\mathbb R)$ since $\V''(\diamond;\lambda)\in C^{3}(\mathbb R)$ by differentiation under the integral sign. Since $j(0;\lambda)=\frac{1}{2}\V''(0;\lambda)> 0$, in a neighborhood $(-\delta,\delta)$ of $x=0$, we may write 
        \begin{align}
            s(x;\lambda)=xk(x;\lambda),\quad k(x;\lambda):=\sqrt{j(x;\lambda)}
            \label{eq:s-x-k}
        \end{align}
        and $k(\diamond;\lambda)$ is $C^{3}((-\delta,\delta))$. It follows\footnote{        
        In fact, $s^{(4)}(0;\lambda)$ exists, since one may compute
        $s'''(x;\lambda)
        =k_1(x;\lambda)+xk'''(x;\lambda)$,  
        where $k_1(\diamond;\lambda)\in C^1(\mathbb R)$. Thus, we have 
        $s^{(4)}(0;\lambda)=k_1'(0;\lambda)+\lim_{x\to 0}k'''(x;\lambda)=k_1'(0;\lambda)+ k'''(0;\lambda)$.  However, continuity of $s^{(4)}(\diamond;\lambda)$ at $x=0$ does not follow from the hypotheses.} that $s(\diamond;\lambda)\in C^{3}((-\delta,\delta))$, so that $s(\diamond;\lambda)\in C^{3}(\mathbb R)$. Continuity of the mixed partials $\partial_\lambda^m\partial_x^ns(x;\lambda)$ near $(0;0)$ then follows from the corresponding property of $\V''(x;\lambda)$ and the formul\ae\ \eqref{eq:V-xsquared-j}--\eqref{eq:s-x-k}.  The corresponding properties of the inverse function $x(s;\lambda)$ follow by (repeated) differentiation of the identity $s(x(s;\lambda);\lambda)=s$ with respect to $s$ and $\lambda$ and the use of $s'(x;\lambda)>0$. 
    \end{proof}
Then using the inverse function as a substitution, for $\lambda>0$ and $x>x_+(\lambda)$ we have
\begin{equation}
    \int_{x_+(\lambda)}^x\sqrt{\V(\bar{x};\lambda)-\lambda}\,\dd\bar{x} = \int_{\sqrt{\lambda}}^{s(x;\lambda)}\sqrt{\bar{s}^2-\lambda}\,x'(\bar{s};\lambda)\,\dd \bar{s}=\lambda\int_1^{s(x;\lambda)/\sqrt{\lambda}}\sqrt{\tau^2-1}\,x'(\sqrt{\lambda}\tau;\lambda)\,\dd\tau.
\end{equation}
But under the further substitution $w=\zeta(\tau)$ we have $w^{1/2}\dd w = \sqrt{\tau^2-1}\,\dd\tau$ for $w>0$ and $\tau>1$ according to \eqref{eq:t-zeta-ODE} from Appendix~\ref{sec:zeta-t-properties}, so 
\begin{equation}
    \int_{x_+(\lambda)}^x\sqrt{\V(\bar{x};\lambda)-\lambda}\,\dd\bar{x} =\lambda\int_0^{\zeta(s(x;\lambda)/\sqrt{\lambda})} w^{1/2}x'(\sqrt{\lambda} t(w);\lambda)\,\dd w.
\end{equation}
Finally, we make one more change of integration variable by $w=\zeta(s(x;\lambda)/\sqrt{\lambda})q^{2/3}$ to arrive at
\begin{equation}
    \int_{x_-(\lambda)}^x\sqrt{\V(\bar{x};\lambda)-\lambda}\,\dd\bar{x}=\frac{2}{3}\lambda \zeta\left(\frac{s(x;\lambda)}{\sqrt{\lambda}}\right)^{3/2}\int_0^1x'\left(\sqrt{\lambda}t\left(\zeta\left(\frac{s(x;\lambda)}{\sqrt{\lambda}}\right)q^{2/3}\right);\lambda\right)\,\dd q,\quad x>x_+(\lambda),\; \lambda>0.
\end{equation}
Therefore multiplying \eqref{eq:general-Langer} by $-3/(4b)$, raising both sides to the $2/3$ power, and applying $t(\diamond)$ to invert $\zeta(\diamond)$ yields the explicit form
\begin{equation}
    y=g(x;\lambda)=2\sqrt{-b}t\left(\zeta\left(\frac{s(x;\lambda)}{\sqrt{\lambda}}\right)\left[\frac{\lambda}{-2b}\int_0^1x'\left(\sqrt{\lambda}t\left(\zeta\left(\frac{s(x;\lambda)}{\sqrt{\lambda}}\right)q^{2/3}\right);\lambda\right)\,\dd q\right]^{2/3}\right).
\label{eq:Langer-explicit}
\end{equation}
Note that since $s\mapsto x(s;\lambda)$ is monotone increasing and $\lambda/b<0$, the $2/3$ power is here applied to a strictly positive quantity.
Although we assumed that $x>x_+(\lambda)$, one can verify using the analytic continuation of $\zeta$ to $-1<t<1$ given explicitly by \eqref{eq:zeta-less-than-1} in Appendix~\ref{sec:zeta-t-properties}
that \eqref{eq:Langer-explicit} in fact holds for all $x>x_-(\lambda)$.  To simplify the notation further, we first introduce an ``inner variable'' by 
\begin{equation}
    \xi:=\zeta\left(\frac{s(x;\lambda)}{\sqrt{\lambda}}\right)\implies \frac{\dd\xi}{\dd x}=\zeta'\left(\frac{s(x;\lambda)}{\sqrt{\lambda}}\right)\frac{s'(x;\lambda)}{\sqrt{\lambda}} = \zeta'(t(\xi))\frac{s'(x;\lambda)}{\sqrt{\lambda}} = \frac{s'(x;\lambda)}{\sqrt{\lambda}t'(\xi)},
\label{eq:inner-variable}
\end{equation}
where we used the identity $\zeta(t(\xi))=\xi$.
Thus, $\xi=O(1)$ as $\lambda\downarrow 0$ means that $x=O(\sqrt{\lambda})$, so the inner variable $\xi$ blows up the interval $x_-(\lambda)<x<x_+(\lambda)$ to fixed size no matter how proximal the turning points are.  Let us also introduce a constant parameter by
\begin{equation}
    \phi(\lambda):=\frac{\lambda}{-b}>0,\quad b=b(\lambda).
\label{eq:phi-define}
\end{equation}
Finally, we introduce a quantity $I(\xi;\sigma)$ by
\begin{equation}
    I(\xi;\sigma):=\frac{1}{2}\phi(\sigma^2)\int_0^1 x'\left(\sigma t(\xi q^{2/3});\sigma^2\right)\,\dd q.
\label{eq:I-define}
\end{equation}
In terms of this notation, we may write \eqref{eq:Langer-explicit} in the form
\begin{equation}
    g(x;\lambda)=\frac{2\lambda^{1/2}}{\phi(\lambda)^{1/2}}t(\xi I(\xi;\lambda^{1/2})^{2/3}),\quad x>x_-(\lambda).
    \label{eq:g-compact}
\end{equation}
We can get an equally-compact expression for the derivative $g'(x;\lambda)$.  Indeed, differentiating $I(\xi;\sigma)$ directly gives
\begin{equation}
    I'(\xi;\sigma)=\frac{1}{2}\sigma\phi(\sigma^2)\int_0^1 x''\left(\sigma t(\xi q^{2/3});\sigma^2\right) t'(\xi q^{2/3})q^{2/3}\,\dd q.
    \label{eq:Iprime-direct}
\end{equation}
However, we also see by direct computation that for $\xi$ fixed,
\begin{equation}
    \frac{\dd}{\dd q}x'\left(\sigma t(\xi q^{2/3});\sigma^2\right)=x''\left(\sigma t(\xi q^{2/3});\sigma^2\right)\sigma t'(\xi q^{2/3})\frac{2}{3}\xi q^{-1/3}.
\end{equation}
Combining these results gives
\begin{equation}
    I'(\xi;\sigma)=\frac{1}{2}\phi(\sigma^2)\frac{3}{2\xi}\int_0^1q\frac{\dd}{\dd q}x'\left(\sigma t(\xi q^{2/3});\sigma^2\right)\,\dd q.
\end{equation}
Hence, integrating by parts,
\begin{equation}
    I'(\xi;\sigma)=\frac{1}{2}\phi(\sigma^2)\frac{3}{2\xi}\left[x'\left(\sigma t(\xi);\sigma^2\right)-
    \int_0^1 x'\left(\sigma t(\xi q^{2/3});\sigma^2\right)\,\dd q\right].
\end{equation}
Therefore, setting $\sigma=\lambda^{1/2}$,
\begin{equation}
    I'(\xi;\lambda^{1/2}) = \frac{3}{4\xi}\phi(\lambda)x'\left(\lambda^{1/2} t(\xi);\lambda\right)-\frac{3}{2\xi}I(\xi;\lambda^{1/2}).
\end{equation}
Going further, \eqref{eq:inner-variable} implies $\lambda^{1/2}t(\xi)=s(x;\lambda)$ since $t$ is the inverse function to $\zeta$. This shows that for any differentiable function $F(\diamond)$,
\begin{equation}
    \frac{\dd}{\dd \xi}F(\xi I(\xi;\lambda^{1/2})^{2/3}) = F'(\xi I(\xi;\lambda^{1/2})^{2/3})\cdot\frac{1}{2}\phi(\lambda)x'(s(x;\lambda);\lambda)I(\xi;\lambda^{1/2})^{-1/3}.
    \label{eq:general-identity}
\end{equation}
Applying this in the case $F(\diamond)=t(\diamond)$ to \eqref{eq:g-compact}, using the chain rule in the form $\dd \xi/\dd x = s'(x;\lambda)/(\lambda^{1/2}t'(\xi))$, and observing that $s'(x;\lambda)x'(s(x;\lambda);\lambda)=1$, we obtain 
\begin{equation}
    g'(x;\lambda)=\frac{\phi(\lambda)^{1/2}}{I(\xi;\lambda^{1/2})^{1/3}}\frac{t'(\xi I(\xi;\lambda^{1/2})^{2/3})}{t'(\xi)}.
    \label{eq:gprime-simple}
\end{equation}

The formul\ae\ \eqref{eq:g-compact} and \eqref{eq:gprime-simple} are especially useful for analyzing the Langer transformation for $x_-(\lambda)<x<\delta$ for some fixed $\delta>0$ when $\lambda$ is small.  However, another advantage 
the formul\ae\ \eqref{eq:g-compact} and \eqref{eq:gprime-simple} enjoy over \eqref{eq:g-x-large} is that they hide completely the turning points $x=x_\pm(\lambda)$, which otherwise can complicate the analysis.  For instance, we have the following result.
\begin{lemma}[Smoothness of the Langer transformation at fixed $\lambda$]
    Suppose that $\V$ satisfies the hypotheses of Theorem~\ref{thm:Schroedinger-BS}.  Then for each $\lambda>0$, the Langer transformation $g(\diamond;\lambda)$ is of class $C^3((x_-(\lambda),\infty))$ with range $(-2\sqrt{-b},\infty)$.
\end{lemma}
\begin{proof}
First, we note that since $\V(\diamond;\lambda)$ is increasing away from the minimizer $x=0$, $\V(x;\lambda)-\lambda$ is bounded below as $x\to +\infty$, so the integral on the right-hand side of \eqref{eq:general-Langer} grows at least linearly in $x$ in this limit; hence also $y=g(x;\lambda)$ grows without bound; together with the condition $g(x_-(\lambda);\lambda)=-2\sqrt{-b}$, this shows that $g(\diamond;\lambda)$ maps $(x_-(\lambda),+\infty))$ to $(-2\sqrt{-b},\infty)$.
    Consider now $g'(\diamond;\lambda)$ in the form \eqref{eq:gprime-simple}. 
    Since by Lemma~\ref{lem:x}, $s(\diamond;\lambda)$ is of class $C^3(\mathbb{R})$ while $\zeta:(-1,\infty)\to (\zeta(-1),\infty)$ is analytic, it follows that $x\mapsto \xi$ is of class $C^3((x_-(\lambda),\infty))$.  As $t:(\zeta(-1),\infty)\to (-1,\infty)$ is analytic and $I(\xi;\lambda^{1/2})>0$, to prove that $g'(\diamond;\lambda)\in C^2$ it therefore remains to observe that $\xi\mapsto I(\xi;\lambda^{1/2})$ is of class $C^2((\zeta(-1),\infty))$ which follows because $x'(\diamond;\lambda)$ is of class $C^2$.  Hence also $g(\diamond;\lambda)\in C^3((x_-(\lambda),+\infty))$.
\end{proof}

\section{Derivatives of the Langer transformation for small $\lambda>0$}
\label{sec:derivatives}
First, we analyze the function $\phi$ defined by \eqref{eq:phi-define} for small $\lambda>0$.  
\begin{lemma}[Differentiability of $\lambda\mapsto\phi(\lambda)$]
Suppose that $\V$ satisfies the hypotheses of Theorem~\ref{thm:Schroedinger-BS}.  Then for some $\delta>0$, $\lambda\mapsto \phi(\lambda)$ is of class $C^1([0,\delta))$ and $\phi(\lambda)=\sqrt{2\V''(0;0)}+O(\lambda)$ as $\lambda\downarrow 0$.
    \label{lem:b}
\end{lemma}
\begin{proof}
    With the substitution $x=x(\lambda^{1/2}w;\lambda)$ in the expression \eqref{eq:b-of-lambda} for $b$, we get
    \begin{equation}
        \frac{1}{\phi(\lambda)}=-\frac{b}{\lambda}=\frac{1}{\pi}\int_{-1}^1\sqrt{1-w^2}\,x'(\lambda^{1/2}w;\lambda)\,\dd w = \frac{1}{\pi}\int_{-1}^1\sqrt{1-w^2}\,p(\lambda^{1/2}w;\lambda)\,\dd w,
        \label{eq:phi-reciprocal}
    \end{equation}
    where $p(s;\lambda):=\frac{1}{2}(x'(s;\lambda)+x'(-s;\lambda))$ is an even function of $s$ for each $\lambda\ge 0$ for which Lemma~\ref{lem:x} yields that $p''(s;\lambda)$ and $p_\lambda(s;\lambda)$ are both continuous on $(-\nu,\nu)\times [0,\lambda_\mathrm{max}]$.  
    For $\lambda>0$, by differentiation under the integral sign, and the use of $p'(s;\lambda)=s\int_0^1p''(s u;\lambda)\,\dd u$
    \begin{equation}
        \frac{\dd}{\dd\lambda}\frac{1}{\phi(\lambda)}=\frac{1}{2\pi}\int_{-1}^1\sqrt{1-w^2}\,w^2\int_0^1 p''(\lambda^{1/2}w u;\lambda)\,\dd u\,\dd w+\frac{1}{\pi}\int_{-1}^1\sqrt{1-w^2}\, p_\lambda(\lambda^{1/2}w;\lambda)\,\dd w,
    \end{equation}
    which is obviously continuous for $\lambda>0$.  Taking the limit $\lambda\downarrow 0$ using uniform convergence of the integrands to take the limit under the integral sign gives
    \begin{equation}
        \lim_{\lambda\downarrow 0}\frac{\dd}{\dd\lambda}\frac{1}{\phi(\lambda)} = \frac{p''(0;0)}{2\pi}\int_{-1}^1\sqrt{1-w^2}\,w^2\,\dd w+p_\lambda(0;0).  
        \label{eq:deriv-phi-reciprocal-limit}
        \end{equation}
    On the other hand, the (right) derivative at $\lambda=0$ is by definition
    \begin{equation}
        \left.\frac{\dd}{\dd\lambda}\frac{1}{\phi(\lambda)}\right|_{\lambda=0}=\lim_{\lambda\downarrow 0}\frac{1}{\pi\lambda}\int_{-1}^1\sqrt{1-w^2}\left(p(\lambda^{1/2}w;\lambda)-p(0;0)\right)\,\dd w.
        \label{eq:phi-inv-deriv-at-0}
    \end{equation}
    Writing
    \begin{equation}
    \begin{split}
        p(\lambda^{1/2}w;\lambda)-p(0;0)&=\left[p(\lambda^{1/2}w;\lambda)-p(0;\lambda)\right]+\left[p(0;\lambda)-p(0;0)\right]\\
        &=\lambda^{1/2}w\int_0^1p'(\lambda^{1/2}w u;\lambda)\,\dd u +\lambda\int_0^1 p_\lambda(0;\lambda u)\,\dd u\\
        &=\lambda w^2\int_0^1 u\int_0^1 p''(\lambda^{1/2}w u v;\lambda)\,\dd v\,\dd u +\lambda\int_0^1 p_\lambda(0;\lambda u)\,\dd u,
        \end{split}
    \end{equation}
    we take the limit in \eqref{eq:phi-inv-deriv-at-0} under the integral sign and obtain the same result as  
    \eqref{eq:deriv-phi-reciprocal-limit}.
    Hence $\lambda\mapsto 1/\phi(\lambda)$ is $C^1([0,\delta))$.  Setting $\lambda=0$ in \eqref{eq:phi-reciprocal} gives $1/\phi(0)=\frac{1}{2}x'(0;0)$.  By implicit differentiation of $s^2=\V(x(s;0);0)$ we obtain $x'(0;0)=\sqrt{2/\V''(0;0)}>0$, so it follows that also $\phi\in C^1([0,\delta))$ and $\phi(\lambda)=\sqrt{2\V''(0;0)}+O(\lambda)$.
\end{proof}

\subsection{Estimates valid for $x\ge\delta$ as $\lambda\downarrow 0$}
\label{sec:OuterEstimates}
We refer to the limit $\lambda\downarrow 0$ for $x\ge\delta$ as \emph{outer asymptotics}.
\begin{lemma}[Outer asymptotics of $J$]
    Suppose that $\V(x;\lambda)$ satisfies the hypotheses of Theorem~\ref{thm:Schroedinger-BS}, and fix $\delta>0$.  Then in the limit $\lambda\downarrow 0$, the quantity $J(x;\lambda)$ defined in \eqref{eq:g-x-large} satisfies
    \begin{equation}
J(x;\lambda)=\left(\int_0^x\sqrt{\V(\bar{x};\lambda)}\,\dd\bar{x}\right)(1+ O(\lambda\ln(\lambda^{-1})) ),
\label{eq:J-expand}
    \end{equation}
    uniformly for $x\ge\delta$.
\label{lem:J}
\end{lemma}
\begin{proof}
    Note that for $\lambda>0$ sufficiently small, we have $x_+(\lambda)<\delta$.  We calculate
    \begin{equation}
    \begin{split}
        J(x;\lambda)-\int_0^x\sqrt{\V(\bar{x};\lambda)}\,\dd\bar{x} &= \int_{x_+(\lambda)}^x\left[\sqrt{\V(\bar{x};\lambda)-\lambda}-\sqrt{\V(\bar{x};\lambda)}\right]\,\dd \bar{x} - \int_0^{x_+(\lambda)}\sqrt{\V(\bar{x};\lambda)}\,\dd\bar{x}\\
        &=-\lambda\int_{x_+(\lambda)}^x\frac{\dd\bar{x}}{\sqrt{\V(\bar{x};\lambda)-\lambda}+\sqrt{\V(\bar{x};\lambda)}} -\int_0^{x_+(\lambda)}\sqrt{\V(\bar{x};\lambda)}\,\dd\bar{x}.
    \end{split}
    \end{equation}
    The hypotheses on $\V$ imply that $x_+(\lambda)\lesssim\lambda^{1/2}$ as $\lambda\downarrow 0$ while $0<\sqrt{\V(\bar{x};\lambda)}\lesssim \bar{x}$ as $\bar{x}\downarrow 0$, so for the second term we have
    \begin{equation}
        0<\int_0^{x_+(\lambda)}\sqrt{\V(\bar{x};\lambda)}\,\dd\bar{x} \lesssim\lambda
        \label{eq:second-term}
    \end{equation}
    (independent of $x$).  Using $\sqrt{\V(\bar{x};\lambda)-\lambda}>0$ for $\bar{x}>x_+(\lambda)$:
    \begin{equation}
        0<\int_{x_+(\lambda)}^x\frac{\dd\bar{x}}{\sqrt{\V(\bar{x};\lambda)-\lambda}+\sqrt{\V(\bar{x};\lambda)}}\le\int_{x_+(\lambda)}^\delta\frac{\dd\bar{x}}{\sqrt{\V(\bar{x};\lambda)}} +\int_\delta^x\frac{\dd\bar{x}}{\sqrt{\V(\bar{x};\lambda)}}.
    \end{equation}
    Again invoking the hypotheses on $\V$, $x_+(\lambda)\gtrsim \lambda^{1/2}$ and $\sqrt{\V(\bar{x};\lambda)}\gtrsim\bar{x}$ as $\bar{x}\downarrow 0$, so
    \begin{equation}
        0<\int_{x_+(\lambda)}^\delta\frac{\dd\bar{x}}{\sqrt{\V(\bar{x};\lambda)}} \lesssim\ln(\lambda^{-1})
        \label{eq:first-term-1}
    \end{equation}
    as $\lambda\downarrow 0$ (independent of $x$).  
    Since the hypothesis $\V(x;\lambda)\gtrsim 1$ as $x\to\infty$ implies also that $\V(x;\lambda)\gtrsim \V(x;\lambda)^{-1}$ holds for all $x\ge\delta$,
    we obtain
    \begin{equation}
        \int_\delta^x\frac{\dd\bar{x}}{\sqrt{\V(\bar{x};\lambda)}}\lesssim \int_\delta^x\sqrt{\V(\bar{x};\lambda)}\,\dd\bar{x}<\int_0^x\sqrt{\V(\bar{x};\lambda)}\,\dd\bar{x}.
        \label{eq:first-term-2}
    \end{equation}
    Finally, since $\lambda\lesssim \lambda\ln(\lambda^{-1})$ and 
    \begin{equation}
        \int_0^x\sqrt{\V(\bar{x};\lambda)}\,\dd\bar{x}\ge\int_0^\delta\sqrt{\V(\bar{x};\lambda)}\,\dd\bar{x},
    \end{equation}
    all three estimates \eqref{eq:second-term}, \eqref{eq:first-term-1}, and \eqref{eq:first-term-2} can be combined into the desired relative error estimate.
\end{proof}

\begin{corollary}[Outer asymptotics of the Langer transformation]
    Suppose that $\V(x;\lambda)$ satisfies the hypotheses of Theorem~\ref{thm:Schroedinger-BS}.  Then as $\lambda\downarrow 0$, 
    \begin{equation}
        y=g(x;\lambda) =g_0(x)\left(1+O(\lambda\ln(\lambda^{-1}))\right),\quad g_0(x):=2\sqrt{\int_0^x\sqrt{\V(\bar{x})}\,\dd\bar{x}},
        \label{eq:outer-limit}
    \end{equation}
    holds uniformly for $x\ge\delta$.
    Also, for $x\ge\delta$ and $\lambda>0$ sufficiently small, we have the lower bound
    \begin{equation}
    g'(x;\lambda)\gtrsim y^{(p-2)/(p+2)},
\label{eq:gp-outer-lower}
\end{equation}
and the upper bounds
\begin{equation}
    g''(x;\lambda)\lesssim y^{(p-6)/(p+2)},\quad g'''(x;\lambda)\lesssim y^{(p-10)/(p+2)}.
    \label{eq:gpp-gppp-outer-upper}
\end{equation}
Here $p=p_+$ is the growth exponent of $\mathcal{V}(x;\lambda)$ as $x\to+\infty$.
\label{cor:outer}
\end{corollary}
\begin{proof}
    We write \eqref{eq:g-x-large} in the form
    \begin{equation}
        g(x;\lambda)=2\sqrt{-b}t(m),\quad m:=\left[\frac{-3}{4b}J(x;\lambda)\right]^{2/3},\quad\frac{\dd m}{\dd x}=-\frac{1}{2b}m^{-1/2}J'(x;\lambda).
    \end{equation}
    Repeated differentiation therefore yields
    \begin{equation}
        g'(x;\lambda)=\frac{1}{\sqrt{-b}}\frac{t'(m)}{m^{1/2}}J'(x;\lambda),
    \end{equation}
    \begin{equation}
        g''(x;\lambda)=\frac{1}{\sqrt{-b}}\frac{t'(m)}{m^{1/2}}J''(x;\lambda) +\frac{1}{4(-b)^{3/2}}\left[2\frac{t''(m)}{m}-\frac{t'(m)}{m^{2}}\right]J'(x;\lambda)^2,
        \label{eq:gpp-outer}
    \end{equation}
    and
    \begin{equation}
    \begin{split}
        g'''(x;\lambda)&= \frac{1}{\sqrt{-b}}\frac{t'(m)}{m^{1/2}}J'''(x;\lambda)+\frac{3}{4(-b)^{3/2}}\left[2\frac{t''(m)}{m}-\frac{t'(m)}{m^{2}}\right]J'(x;\lambda)J''(x;\lambda)\\
        &\quad\quad+\frac{1}{8(-b)^{5/2}}\left[2\frac{t'''(m)}{m^{3/2}}-3\frac{t''(m)}{m^{5/2}}+2\frac{t'(m)}{m^{7/2}}\right]J'(x;\lambda)^3.
    \end{split}
    \label{eq:gppp-outer}
    \end{equation}    
    According to Lemma~\ref{lem:J}, the condition $x\ge\delta$ bounds $J(x;\lambda)$ away from zero uniformly for small $\lambda$, and by Lemma~\ref{lem:b}, we have $b=-\lambda/\phi(\lambda)=O(\lambda)$ as $\lambda\downarrow 0$.  Hence $m\gtrsim\lambda^{-2/3}$, so we can apply to \eqref{eq:g-x-large} and its first three derivatives with respect to $x$ the asymptotic approximations of derivatives of the analytic function $t(\diamond)$ given in \eqref{eq:t-large-z} from Appendix~\ref{sec:zeta-t-properties}.
    Using $m^{-1}=O(\lambda^{2/3})$ which dominates the relative error term in \eqref{eq:J-expand},  
it follows from Lemma~\ref{lem:J} that
\begin{equation}
    g(x;\lambda)=2\sqrt{-b}\left(\frac{4}{3}\right)^{1/2}m^{3/4}(1+O(m^{-1})) = g_0(x)(1+O(\lambda^{2/3})).
\end{equation}
Since the hypotheses on $\V$ imply $g_0(x)=(2/(p+2))^{1/2}\V_+(\lambda)^{1/4}x^{p/4+1/2}(1+o(1))$ as $x\to+\infty$, with $y=g(x;\lambda)$ we deduce that 
\begin{equation}
 x=\left(1+\frac{1}{2}p\right)^{2/(p+2)}\V_+(\lambda)^{-1/(p+2)}y^{4/(p+2)}\left(1+O(\lambda^{2/3})+o(1)\right)   
 \label{eq:x-y}
\end{equation}
with the term $O(\lambda^{2/3})$ being uniform for $x\ge\delta$ and $o(1)$ referring to the limit $x\to+\infty$.
Now, for $x\ge\delta$, $\sqrt{\V(x;\lambda)-\lambda}=\sqrt{\V(x;\lambda)}(1+O(\lambda))$, so the derivatives of $J(x;\lambda)$ are
\begin{equation}
\begin{split}
    J'(x;\lambda)&=\sqrt{\V(x;\lambda)-\lambda}=\V(x;\lambda)^{1/2}(1+O(\lambda)),\\ 
    J''(x;\lambda)&=\frac{\V'(x;\lambda)}{2\sqrt{\V(x;\lambda)-\lambda}}=\frac{\V'(x;\lambda)}{2\V(x;\lambda)^{1/2}}(1+O(\lambda)),\\ J'''(x;\lambda)&=\frac{\V''(x;\lambda)}{2\sqrt{\V(x;\lambda)-\lambda}}-\frac{\V'(x;\lambda)^2}{4(\V(x;\lambda)-\lambda)^{3/2}} = \frac{\V''(x;\lambda)}{2\V(x;\lambda)^{1/2}}(1+O(\lambda))+\frac{\V'(x;\lambda)^2}{4\V(x;\lambda)^{3/2}}(1+O(\lambda)).
\end{split}
\label{eq:J-derivs}
\end{equation}
 so similarly,
\begin{equation}
    g'(x;\lambda)=\frac{1}{\sqrt{-b}}\left(\frac{3}{4}\right)^{1/2}m^{-3/4}(1+O(m^{-1}))\sqrt{\V(x;\lambda)-\lambda} = g_0'(x)(1+O(\lambda^{2/3})).
\end{equation}
Using \eqref{eq:x-y} establishes the lower bound \eqref{eq:gp-outer-lower}.
The hypotheses on $\V$ and \eqref{eq:J-derivs} also imply that $J''(x;\lambda)\lesssim x^{p/2-1}$ and $J'''(x;\lambda)\lesssim x^{p/2-2}$ for $x\ge\delta$ and $\lambda>0$ sufficiently small.  Combining these estimates with \eqref{eq:gpp-outer}, \eqref{eq:gppp-outer}, and \eqref{eq:t-large-z} together with Lemma~\ref{lem:J} and $b\gtrsim\lambda$ from Lemma~\ref{lem:b} gives
$g''(x;\lambda)\lesssim x^{p/4-3/2}$ and $g'''(x;\lambda)\lesssim x^{p/4-5/2}$ for $x\ge\delta$ and $\lambda>0$ sufficiently small.  Combining these with \eqref{eq:x-y} yields \eqref{eq:gpp-gppp-outer-upper} and completes the proof.
\end{proof}

\subsection{Estimates valid for $\frac{1}{2}x_-(\lambda)<x<\delta$ as $\lambda\downarrow 0$}
\label{sec:InnerEstimates}
Since $x=x_-(\lambda)$ corresponds to $t=-1$, which is the largest real value at which $\zeta(t)$ defined in Appendix~\ref{sec:zeta-t-properties} fails to be analytic, it is useful to bound $t$ away from $-1$ and hence we consider $x>\frac{1}{2}x_-(\lambda)$.  The upper bound of $x<\delta$ then reaches the lower bound on $x$ where the analysis of Section~\ref{sec:OuterEstimates} holds.  
\begin{lemma}[Expansion of $I$]
Suppose $\V(x;\lambda)$ satisfies the hypotheses of Theorem~\ref{thm:Schroedinger-BS}. Then $\xi\mapsto I(\xi;\sigma)$ has the expansion
\begin{equation}
\label{eq:Iexpand}
    =1 -\frac{\sqrt{2}\V'''(0;\sigma^2)}{3\V''(0;\sigma^2)^{3/2}}\sigma\int_0^1 t(\xi q^{2/3})\,\dd q + O\left(\sigma^2\langle \xi\rangle^{3/2}\right),\quad \sigma\langle \xi\rangle^{3/4}\lesssim 1,\quad\langle \diamond\rangle:=\sqrt{1+\diamond^2}.
\end{equation}
    \label{lem:I-expand}
\end{lemma}
\begin{proof}
By Lemma~\ref{lem:b} we have $\phi(\sigma^2)=\sqrt{2\V''(0;0)}+O(\sigma^2)$,
            and Lemma~\ref{lem:x} implies that $s\mapsto x'(s;\sigma^2)$ has two continuous derivatives in a neighborhood of $s=0$.  By repeated implicit differentiation of $s^2=\V(x(s;\sigma^2);\sigma^2)$, it has the Taylor representation
    \begin{equation}
        x'(s;\sigma^2)
        =x'(0;\sigma^2)+x''(0;\sigma^2)s +\frac{1}{2}x'''(c;\sigma^2)s^2
        = \sqrt{\frac{2}{\V''(0;\sigma^2)}}-\frac{2\V'''(0;\sigma^2)}{3\V''(0;\sigma^2)^2}s + \frac{1}{2}x'''(c;\sigma^2)s^2
        \label{eq:x'-Taylor}
    \end{equation}
    holding for some number $c$ between $0$ and $s$.  Due to the asymptotic behavior of $t(\diamond)$ given in \eqref{eq:t-large-z} from Appendix~\ref{sec:zeta-t-properties}, for $\xi$ as large as $\xi=O(\sigma^{-4/3})$, we have $\sigma t(\xi q^{2/3})=O(1)$ uniformly for $q\in (0,1)$, so the result follows from integration term-by-term in \eqref{eq:I-define}.
\end{proof}

With this, we can prove the following lower bound on $g'(x;\lambda)$ valid near the turning points.  The \emph{inner region} refers to the domain $\frac{1}{2}x_-(\lambda)<x<\delta$ with $\lambda>0$ sufficiently small.
\begin{lemma}[Lower bound on $g'(x;\lambda)$ in the inner region]
    Suppose $\V(x;\lambda)$ satisfies the hypotheses of Theorem~\ref{thm:Schroedinger-BS}.  If $\delta>0$ is sufficiently small, $g'(x;\lambda)\gtrsim 1$ holds for $\frac{1}{2}x_-(\lambda)<x<\delta$ and $\lambda>0$ sufficiently small.  
    \label{lem:gp-lower-bound-inner}
\end{lemma}
\begin{proof}
By Lemma~\ref{lem:I-expand}, if $\lambda^{1/2}\langle \xi\rangle^{3/4}$ is sufficiently small, which is guaranteed by the condition $x_-(\lambda)<x<\delta$ for sufficiently small $\delta>0$, then we will have $\frac{1}{2}\le I(\xi;\lambda^{1/2})\le 2$.  Using the estimates 
$\langle\zeta\rangle^{-1/4}\lesssim    t'(\zeta)\lesssim \langle\zeta\rangle^{-1/4}$ valid for $\zeta\ge \zeta(-1)+\delta$,
using the representation \eqref{eq:gprime-simple} of $g'(x;\lambda)$, and recalling $\phi(\lambda)\to \sqrt{2\V''(0;0)}$ as $\lambda\downarrow 0$ from Lemma~\ref{lem:b} completes the proof.
\end{proof}

Now we turn to $g''(x;\lambda)$.  Recalling the inner variable $\xi$ defined by \eqref{eq:inner-variable}, direct differentiation of \eqref{eq:gprime-simple} using the chain rule and \eqref{eq:general-identity} with $F(\diamond)=t'(\diamond)$ yields
\begin{equation}
    g''(x;\lambda)=\frac{s'(x;\lambda)}{\lambda^{1/2}}\frac{\phi(\lambda)^{1/2}}{t'(\xi)}\left(-\frac{I'(\xi;\lambda^{1/2})}{3I(\xi;\lambda^{1/2})^{4/3}}\frac{t'(\xi I(\xi;\lambda^{1/2})^{2/3})}{t'(\xi)} 
+ \frac{N(\xi;\lambda^{1/2})
    }{I(\xi;\lambda^{1/2})^{1/3}t'(\xi)^2
    }\right)
    \label{eq:gpp-again}
\end{equation}
after also using $x'(s(x;\lambda);\lambda)=x'(\lambda^{1/2}t(\xi);\lambda)$.  Here, the numerator of the second term is defined by
\begin{equation}
    N(\xi;\sigma):=\frac{1}{2}\phi(\sigma^2)x'(\sigma t(\xi);\sigma^2)I(\xi;\sigma)^{-1/3}t''(\xi I(\xi;\sigma)^{2/3})t'(\xi)-t'(\xi I(\xi;\sigma)^{2/3})t''(\xi).
\end{equation}
    To consider $\lambda>0$ small, note first that the factor of $\lambda^{-1/2}$is explicitly cancelled for the first term by using \eqref{eq:Iprime-direct}.  For the second term, one combines $\phi(\lambda)\to \sqrt{2\V''(0;0)}>0$ as $\lambda\downarrow 0$ and $x'(0;0)=\sqrt{2/\V''(0;0)}$ from Lemma~\ref{lem:b} and its proof to get from \eqref{eq:I-define} that $I(\xi;\sigma)\to 1$ as $\sigma\downarrow 0$.  Hence also $N(\xi;\sigma)\to 0$ as $\sigma\downarrow 0$ for each fixed $\xi>\zeta(-1)$, so that
\begin{equation}    N(\xi;\lambda^{1/2})=\int_0^{\lambda^{1/2}}N_\sigma(\xi;\sigma)\,\dd\sigma = \lambda^{1/2}\int_0^1 N_\sigma(\xi;\lambda^{1/2}r)\,\dd r.
\end{equation}
We therefore obtain a formula for $g''(x;\lambda)$ without division by $\lambda^{1/2}$:
\begin{equation}
\begin{split}
g''(x;\lambda)&=\frac{s'(x)
\phi(\lambda)^{1/2}}{t'(\xi)}\left(-\frac{\phi(\lambda)}{6I(\xi;\lambda^{1/2})^{4/3}}\frac{t'(\xi I(\xi;\lambda^{1/2})^{2/3})}{t'(\xi)}\int_0^1x''\left(\lambda^{1/2}t(\xi q^{2/3})\right)t'(\xi q^{2/3})q^{2/3}\,\dd q \right.\\
    &\qquad \left.\vphantom{-\frac{\phi(\lambda)}{6I(\xi;\lambda^{1/2})^{4/3}}\frac{t'(\xi I(\xi;\lambda^{1/2})^{2/3})}{t'(\xi)}}+\frac{1}{I(\xi;\lambda^{1/2})^{1/3}t'(\xi)^2}\int_0^1 N_\sigma(\xi;\lambda^{1/2} r)\,\dd r\right).
    \end{split}    
\label{eq:gpp-formula}
\end{equation}

With this representation, one can easily check that $g''(x;\lambda)$ has a limit as $\lambda\downarrow 0$ with the inner variable $\xi$ held fixed (the ``inner limit''), which also forces $x\to 0$ since $x=O(\lambda^{1/2})$ for $\xi$ fixed.  Indeed, since $x(0;\lambda)=0$, $x''(s;\lambda)$ is continuous at $s=0$, and $I(\xi;\lambda^{1/2})\to 1$ in this limit, using uniform convergence of the integrand to exchange the limit and the integral,
\begin{equation}
    \lim_{\lambda\downarrow 0}g''(x;\lambda) =
    \frac{s'(0;0)\phi(0)^{1/2}}{t'(\xi)}\left(-\frac{\phi(0)}{6}x''(0;0)\int_0^1t'(\xi q^{2/3})q^{2/3}\,\dd q +\frac{1}{t'(\xi)^2}N_\sigma(\xi;0)\right),\quad \text{$\xi$ fixed.}
    \label{eq:gpp-inner-limit-1}
\end{equation}
Now, $N(\xi;\sigma)$ depends explicitly on $\sigma$ and also implicitly via $I(\xi;\sigma)$, so 
\begin{multline}
    N_\sigma(\xi;\sigma)=\frac{\phi'(\sigma^2)\sigma x'(\sigma t(\xi);\sigma^2)t''(\xi I(\xi;\sigma)^{2/3})t'(\xi)}{I(\xi;\sigma)^{1/3}}+\frac{\phi(\sigma^2)x''(\sigma t(\xi);\sigma^2)t(\xi)t''(\xi I(\xi;\sigma)^{2/3})t'(\xi)}{2I(\xi;\sigma)^{1/3}}\\+\left[-\frac{\phi(\sigma^2)x'(\sigma t(\xi);\sigma^2)t''(\xi I(\xi;\sigma)^{2/3})t'(\xi)}{6I(\xi;\sigma)^{4/3}} +\frac{\phi(\sigma^2)x'(\sigma t(\xi);\sigma^2)t'''(\xi I(\xi;\sigma)^{2/3})\xi t'(\xi)}{3I(\xi;\sigma)^{2/3}}\right.\\
    \left.-\frac{2t''(\xi I(\xi;\sigma)^{2/3})\xi t''(\xi)}{3I(\xi;\sigma)^{1/3}}\right]I_\sigma(\xi;\sigma).
    \label{eq:N-sigma}
\end{multline}
Noting that 
\begin{equation}
I_\sigma(\xi;\sigma)=\sigma\phi'(\sigma^2)\int_0^1 x'(\sigma t(\xi q^{2/3});\sigma^2)\,\dd q +\frac{1}{2}\phi(\sigma^2)\int_0^1 x''(\sigma t(\xi q^{2/3});\sigma^2)t(\xi q^{2/3})\,\dd q,
\label{eq:Itilde-sigma}
\end{equation}
letting $\sigma\downarrow 0$ gives 
\begin{multline}
    N_\sigma(\xi;0)=\frac{\phi(0)x''(0;0)}{2}t(\xi)t'(\xi)t''(\xi) + \left[-\frac{\phi(0)^2x'(0;0)x''(0;0)}{12}t'(\xi)t''(\xi)\right.\\\left.+\frac{\phi(0)^2x'(0;0)x''(0;0)}{6}\xi t'(\xi)t'''(\xi) - \frac{\phi(0)x''(0;0)}{3}\xi t''(\xi)^2\right]\int_0^1 t(\xi q^{2/3})\,\dd q,
    \label{eq:N-sigma-0}
\end{multline}
interchanging limit and integral by uniform convergence once again. 
The inner limit of $g''(x;\lambda)$ is therefore
\begin{multline}
    \lim_{\lambda\downarrow 0}g''(x;\lambda)=\frac{s'(0;0)\phi(0)^{3/2}x''(0;0)}{12t'(\xi)}\left[-2\int_0^1t'(\xi q^{2/3})q^{2/3} \,\dd q +\frac{6t(\xi)t''(\xi)}{t'(\xi)}\right.\\
    \left.+\left(-\phi(0)x'(0;0)\frac{t''(\xi)}{t'(\xi)}+2\phi(0)x'(0;0)\frac{\xi t'''(\xi)}{t'(\xi)} - 4\frac{\xi t''(\xi)^2}{t'(\xi)^2}\right)\int_0^1t(\xi q^{2/3})\,\dd q\right],\quad \text{$\xi$ fixed.}
\end{multline}
Using $\frac{1}{2}\phi(0)x'(0;0)=1$ and $s'(0;0)\phi(0)^{3/2}x''(0;0)=-2^{5/4}\V'''(0;0)/(3\V''(0;0)^{3/4})$, the inner limit simplifies to 
\begin{equation}
    \lim_{\lambda\downarrow 0}g''(x;\lambda)=
    \frac{2^{1/4}\V'''(0;0)}{9\V''(0;0)^{3/4}}K(\xi),\quad\text{$\xi$ fixed,}
    \label{eq:gpp-inner-limit}
\end{equation}
where a function of the inner variable $\xi$ independent of all details of the potential $\V$ is defined by
\begin{equation}
K(\xi):=\frac{1}{t'(\xi)}\int_0^1t(\xi q^{2/3})q^{2/3}\,\dd q -\frac{3t(\xi)t''(\xi)}{t'(\xi)^2} + \left(\frac{t''(\xi)}{t'(\xi)^2}-\frac{2\xi t'''(\xi)}{t'(\xi)^2}+\frac{2\xi t''(\xi)^2}{t'(\xi)^3}\right)\int_0^1t(\xi q^{2/3})\,\dd q.
    \label{eq:H-def}
\end{equation}
With this calculation finished, we can prove the following.
\begin{lemma}[Upper bound on $|g''(x;\lambda)|$ in the inner region]
    Suppose $\V(x;\lambda)$ satisfies the hypotheses of Theorem~\ref{thm:Schroedinger-BS}. If $\delta>0$ is sufficiently small, $|g''(x;\lambda)|\lesssim 1$ holds for $\frac{1}{2}x_-(\lambda)<x<\delta$ and $\lambda>0$ sufficiently small.
\label{lem:gpp-inner-bounded}
\end{lemma}
\begin{proof}
Since $K$ is an analytic function of $\xi>\zeta(-1)$ that has a finite limit as $\xi\to+\infty$ according to \eqref{eq:t-large-z} from Appendix~\ref{sec:zeta-t-properties}, the inner limit of $g''(x;\lambda)$ is uniformly bounded when $x>\frac{1}{2}x_-(\lambda)$, which bounds $\xi$ above $\zeta(-1)$, so it suffices to analyze the difference between $g''(x;\lambda)$ and its inner limit, i.e., the difference between the right-hand sides of \eqref{eq:gpp-formula} and \eqref{eq:gpp-inner-limit-1}.  According to Lemma~\ref{lem:x}, $s'(x;\lambda)$ is continuous at $(0;0)$ so $s'(x;\lambda)-s'(0;0)$ is bounded.  By Lemma~\ref{lem:b}, $\phi(\lambda)$ is continuous at $\lambda=0$ and $\phi(0)>0$, so  $\phi(\lambda)^{1/2}-\phi(0)^{1/2}$ is bounded.  Also, from Lemma~\ref{lem:I-expand} for $\lambda>0$ sufficiently small we get $I(\xi;\lambda^{1/2})-1=O(\delta)$ for $\frac{1}{2}x_-(\lambda)<x<\delta$.  Therefore under the same condition, $t'(\xi I(\xi;\lambda^{1/2})^{2/3})/t'(\xi)-1=O(\delta)$.  
Again using Lemma~\ref{lem:x}, $x''(\lambda^{1/2}t(\xi q^{2/3});\lambda)-x''(0;0)$ is bounded uniformly for $0<q<1$, and since $t'(\diamond)$ is a bounded function according to \eqref{eq:t-large-z} this is enough to control the difference between the integral on the first line of \eqref{eq:gpp-formula} and $x''(0;0)\int_0^1t'(\xi q^{2/3})q^{2/3}\,\dd q$.

It therefore remains to show that $N_\sigma(\xi;\lambda^{1/2}r)-N_\sigma(\xi;0)$ is uniformly bounded as $\lambda\downarrow 0$ for $\xi$ as large as $\lambda^{-2/3}$ and $r\in (0,1)$.  But comparing the right-hand sides of \eqref{eq:N-sigma} (with \eqref{eq:Itilde-sigma}) and of \eqref{eq:N-sigma-0}, this follows under the same conditions from the same estimates, although in replacing $x''(\sigma t(\xi q^{2/3});\sigma^2)$ with $x''(0;0)$ in the second term of $I_\sigma(\xi;\lambda^{1/2}r)$ one must use the fact that growth of $\int_0^1 t(\xi q^{2/3})\,\dd q$ is precisely compensated by the decay of the terms in square brackets in \eqref{eq:N-sigma} as $\xi\to+\infty$.
\end{proof}
Next, we obtain $g'''(x;\lambda)$ 
by differentiation of \eqref{eq:gpp-formula}, resulting in an expression with many terms. Differentiation of the factor $s'(x;\lambda)$ yields a suite of terms equal to $s''(x;\lambda)g''(x;\lambda)/s'(x;\lambda)$ which is bounded for $-\frac{1}{2}x_-(\lambda)<x<\delta$ and $\lambda$ small by Lemma~\ref{lem:gpp-inner-bounded}.
For the remaining terms, one can differentiate with respect to $\xi$ and then account for the chain rule by a factor of $\dd \xi/\dd x = s'(x;\lambda)/(\lambda^{1/2}t'(\xi))$ (see \eqref{eq:inner-variable}).  Two different types of terms are generated:  those that come with a factor that is $O(\lambda^{1/2}t'(\xi))$ and those that do not.  Only the latter terms are of concern in bounding $|g'''(x;\lambda)|$ given the chain-rule factor inversely proportional to $\lambda^{1/2}t'(\xi)$.  In particular, all derivatives falling on $I(\xi;\lambda^{1/2})$ can be ignored, as can derivatives of $x^{(k)}(\lambda^{1/2}rt(\xi q^{2/3});\lambda)$ for $(q,r)\in [0,1]^2$.  For example, in the expression on the first line of the right-hand side of \eqref{eq:gpp-formula} (extracting the factor $-s'(x;\lambda)\phi(\lambda)^{3/2}/6$, which has already been accounted for), only three occurrences of $\xi$ (indicated with $\eta$) will fail to produce $O(\lambda^{1/2}t'(\xi))$ factors upon differentiation:
\begin{multline}
    -\frac{\phi(\lambda)^{3/2}s'(x;\lambda)}{6}\frac{\dd}{\dd \xi}\left[\frac{1}{I(\xi;\lambda^{1/2})^{4/3}}\frac{t'(\xi I(\xi;\lambda^{1/2})^{2/3})}{t'(\xi)^2}\int_0^1x''\left(\sqrt{\lambda}t(\xi q^{2/3});\lambda\right)t'(\xi q^{2/3})q^{2/3}\,\dd q\right]\\
    =-\frac{\phi(\lambda)^{3/2}s'(x;\lambda)}{6}\frac{\dd}{\dd \eta}\left.\left[\frac{1}{I(\xi;\lambda^{1/2})^{4/3}}\frac{t'(\eta I(\xi;\lambda^{1/2})^{2/3})}{t'(\eta)^2}\int_0^1x''\left(\sqrt{\lambda}t(\xi q^{2/3});\lambda\right)t'(\eta q^{2/3})q^{2/3}\,\dd q\right]\right|_{\eta=\xi}\\
    +O(\lambda^{1/2}t'(\xi)).
\end{multline}
Performing the differentiation with respect to $\eta$ and expanding the resulting expression in the inner limit of $\lambda\to 0$ with $\xi$ fixed, further controllable error terms of the same order are generated:
\begin{multline}
       -\frac{\phi(\lambda)^{3/2}s'(x;\lambda)}{6}\frac{\dd}{\dd \xi}\left[\frac{1}{I(\xi;\lambda^{1/2})^{4/3}}\frac{t'(\xi I(\xi;\lambda^{1/2})^{2/3})}{t'(\xi)^2}\int_0^1x''\left(\sqrt{\lambda}t(\xi q^{2/3});\lambda\right)t'(\xi q^{2/3})q^{2/3}\,\dd q\right]\\
=-\frac{\phi(0)^{3/2}s'(0;0)x''(0;0)}{6} \frac{\dd}{\dd \xi}\left[\frac{1}{t'(\xi)}\int_0^1 t'(\xi q^{2/3})q^{2/3}\,\dd q\right] + O(\lambda^{1/2}t'(\xi))\\
=\frac{2^{1/4}\V'''(0;0)}{9\V''(0;0)^{3/4}}\frac{\dd}{\dd \xi}\left[\frac{1}{t'(\xi)}\int_0^1 t'(\xi q^{2/3})q^{2/3}\,\dd q\right] + O(\lambda^{1/2}t'(\xi)).
\end{multline}
We therefore recognize the explicit terms as the derivative with respect to $\xi$ of the contribution to the inner limit of $g''(x;\lambda)$ given in \eqref{eq:gpp-inner-limit} arising from the first term in $K(\xi)$ defined in \eqref{eq:H-def}.  A similar computation applies to the terms on the second line of \eqref{eq:gpp-formula}, with the result that
\begin{equation}
    \frac{\dd}{\dd \xi}g''(x;\lambda)=\frac{2^{1/4}\V'''(0;0)}{9\V''(0;0)^{3/4}}K'(\xi) + O(\lambda^{1/2}t'(\xi)),
\end{equation}
where the error term is uniform for $-\frac{1}{2}x_-(\lambda)<x<\delta$, Therefore, putting in the chain rule factor $\dd \xi/\dd x$ shows that 
\begin{equation}
    g'''(x;\lambda)=\frac{s'(x;\lambda)}{\lambda^{1/2}t'(\xi)}\frac{2^{1/4}\V'''(0;0)}{9\V''(0;0)^{3/4}}K'(\xi) + O(1)
    \label{eq:gppp-estimate}
\end{equation}
as $\lambda\downarrow 0$ uniformly on the interval $x_-(E)<x<\delta$.  
In fact, the explicit term vanishes:
\begin{lemma}[Constancy of the function $K(\xi)$]
    $K(\xi)=1$ holds identically for $\xi>\zeta(-1)$.
\end{lemma}
\begin{proof}
    Denoting $n(\xi):=\int_0^1t(\xi q^{2/3})\,\dd q$, integration by parts shows that $\int_0^1t'(\xi q^{2/3})q^{2/3}\,\dd q=n'(\xi)=\frac{3}{2}\xi^{-1}(t(\xi)-n(\xi))$.  Hence  
    $K(\xi)$ involves only $n(\xi)$, $\xi$, $t(\xi)$, $t'(\xi)$, $t''(\xi)$, and $t'''(\xi)$.  Eliminating $\xi$ via  $\xi= t'(\xi)^2(t(\xi)^2-1)$ (see \eqref{eq:t-zeta-ODE} in Appendix~\ref{sec:zeta-t-properties}), as well as $t''(\xi)$ and $t'''(\xi)$ via its first two derivatives results in 
    \begin{equation}
    K(\xi)-1=-\frac{(2n(\xi)t'(\xi)^3-1)(1+2t(\xi)^2)}{t(\xi)^2-1}.
    \label{eq:Hplus2}
    \end{equation}
    Explicit differentiation and subsequent elimination of $n'(\xi)$, $\xi$, and $t''(\xi)$ shows that
    \begin{equation}
    \begin{split}
        \frac{\dd}{\dd \xi}(K(\xi)-1)&=\frac{3t(\xi)(2+t(\xi)^2)t'(\xi)(2n(\xi)t'(\xi)^3-1)}{(t(\xi)^2-1)^2}\\
        &=-3\frac{3+t(\xi)^2}{(t(\xi)^2-1)(1+2t(\xi)^2)}t'(\xi)(K(\xi)-1).
    \end{split}
    \end{equation}
    If $\xi>0$, then $t(\xi)>1$, so integration yields
    \begin{equation}
    K(\xi)-1=C\left(\frac{t(\xi)+1}{t(\xi)-1}\right)^2\ee^{5\arctan(\sqrt{2}t(\xi))/\sqrt{2}},\quad \xi>0.
    \end{equation}
    Letting $\xi\to+\infty$ so also $t(\xi)\to+\infty$ gives
    \begin{equation}
        \lim_{\xi\to+\infty}
        K(\xi)-1= C\ee^{5\pi/\sqrt{8}}.
    \end{equation}
    On the other hand, using the large-$\xi$ asymptotic of $t'(\xi)$ given in \eqref{eq:t-large-z} from Appendix~\ref{sec:zeta-t-properties}  and  also that 
\begin{equation}
\begin{split}
    n(\xi) = \frac{3}{2}\xi^{-3/2}\int_0^\xi t(w)w^{1/2}\,\dd w&=\frac{3}{2}\xi^{-3/2}\left(\frac{4}{9}\left(\frac{4}{3}\right)^{1/2}\xi^{9/4}(1+O(\xi^{-1}))\right)\\ &=\frac{2}{3}\left(\frac{4}{3}\right)^{1/2}\xi^{3/4}(1+O(\xi^{-1})),\quad\xi\to+\infty
    \end{split}
\end{equation}
as follows from the asymptotic behavior of $t(\xi)$ in \eqref{eq:t-large-z}, we see from \eqref{eq:Hplus2} that in fact $K(\xi)-1\to 0$ as $\xi\to+\infty$, and therefore the integration constant is $C=0$.  Thus $K(\xi)-1=0$ is an identity for $\xi>0$, however the original formula \eqref{eq:H-def} shows that $K(\xi)$ is analytic at $\xi=0$.  We conclude that $K(\xi)=1$ holds for all $\xi>\zeta(-1)$.
\end{proof}
Therefore the inner limit in \eqref{eq:gpp-inner-limit} coincides with the constant $g_0''(0)$ as defined in \eqref{eq:outer-limit}, and we also conclude from \eqref{eq:gppp-estimate} the following.
\begin{lemma}[Upper bound on $|g'''(x;\lambda)|$ in the inner region]
    Suppose $\V(x;\lambda)$ satisfies the hypotheses of Theorem~\ref{thm:Schroedinger-BS}.  If $\delta>0$ is sufficiently small, $|g'''(x;\lambda)|\lesssim 1$ holds for $\frac{1}{2}x_-(\lambda)<x<\delta$ and $\lambda>0$ sufficiently small.
    \label{lem:gppp-inner-bound}
\end{lemma}

\subsection{Application to estimates of $Q(y;\lambda)$}
We now apply the bounds obtained on the Langer transformation to estimate the perturbation term $Q(y)$ given by \eqref{eq:Schroedinger-Q} and \eqref{eq:ZS-Q} in the Schr\"odinger and Zakharov-Shabat cases respectively.
\begin{proposition}[Estimate of $Q(y)$ for the Schr\"odinger equation]
    Suppose $\V(x;\lambda)$ satisfies the hypotheses of Theorem~\ref{thm:Schroedinger-BS}.  If $\lambda_\mathrm{max}>0$ is sufficiently small, there is a constant $C>0$ such that the function $Q(y;\lambda)$ defined by \eqref{eq:Schroedinger-Q} satisfies $|Q(y;\lambda)|\le C/(1+y^2)$ for all $y\in\mathbb{R}$ and $0<\lambda\le\lambda_\mathrm{max}$.
    \label{prop:Schrodinger-Q-estimate}
\end{proposition}

\begin{proof}
    For $y>-2\sqrt{-b}(1-\delta)$, with $0<\delta<1$ fixed, this follows from Corollary~\ref{cor:outer}, and from Lemmas~\ref{lem:gp-lower-bound-inner}, \ref{lem:gpp-inner-bounded}, and \ref{lem:gppp-inner-bound}.  These results were obtained by analyzing the Langer transformation $y=g(x;\lambda)$ based on integration from $x=x_+(\lambda)$ and $y=2\sqrt{-b}$.  However, the definition \eqref{eq:b-of-lambda} of $b$ implies that a completely parallel analysis applies based on integration instead from $x=x_-(\lambda)$ and $y=-2\sqrt{-b}$.  This extends the result to $y\in\mathbb{R}$.
\end{proof}

\begin{proposition}[Estimate of $Q(y)$ for the Zakharov-Shabat system]
    Under the hypotheses of Theorem~\ref{thm:ZS-BS}, for each $\lambda_\mathrm{max}\in (\min r_+(\diamond),1)$ there exists a constant $C$ such that the function $Q(y;\lambda)$ defined by \eqref{eq:ZS-Q} satisfies $|Q(y;\lambda)|\le C/(1+y^2)$ for all $y\in\mathbb{R}$ and $\min r_+(\diamond)<\lambda\le\lambda_\mathrm{max}$.
    \label{prop:ZS-Q-estimate}
\end{proposition}
\begin{proof}
    We first apply Lemma~\ref{lem:ZS-Schroedinger} to write the Langer transformation in the form $y=g(x;\lambda)=\widetilde{g}(\widetilde{x};\lambda)$ where $\widetilde{x}=\widetilde{x}(x;\lambda)$, and note that using $\dd\widetilde{x}/\dd x = R_+(x;\lambda)$ the formula \eqref{eq:ZS-Q} becomes
    \begin{equation}
        Q(y;\lambda)=-\frac{3\widetilde{g}''(\widetilde{x};\lambda)^2}{4\widetilde{g}'(\widetilde{x};\lambda)^4}+\frac{\widetilde{g}'''(\widetilde{x};\lambda)}{2\widetilde{g}'(\widetilde{x};\lambda)^3} +\frac{1}{\widetilde{g}'(\widetilde{x};\lambda)^2}\left[\frac{7R_+'(x;\lambda)^2}{16R_+(x;\lambda)^3}-\frac{R_+''(x;\lambda)}{4R_+(x;\lambda)^2}\right].
    \end{equation}
    But Proposition~\ref{prop:Schrodinger-Q-estimate} applies to the first two terms, since they have the form of \eqref{eq:Schroedinger-Q}, and $\widetilde{g}(\widetilde{x};\lambda)$ is associated with an energy-dependent potential \eqref{eq:ZS-potential} that satisfies the hypotheses of Theorem~\ref{thm:Schroedinger-BS}.  For the remaining terms, if $y$ and hence also $x$ is bounded, then $\widetilde{g}'(\widetilde{x};\lambda)$ is bounded below (Lemma~\ref{lem:gp-lower-bound-inner}) and $R_+(x;\lambda)\ge\sqrt{\min r_+(\diamond)-\max r_-(\diamond)}>0$ is of class $C^5$, so the terms are bounded uniformly as $\lambda\downarrow \min r_+(\diamond)$.  On the other hand, if $y$ becomes large, we can use the fact that $R_+'(x;\lambda)=O(x^{-1})$ and $R_+''(x;\lambda)=O(x^{-2})$ uniformly with respect to $\lambda$ and Corollary~\ref{cor:outer} to see that $O(\widetilde{g}'(\widetilde{x};\lambda)^{-2}x^{-2})=O(y^{-2})$.
\end{proof}
The fact that the bound on $Q(y)$ is essentially the same in both cases allows for a streamlined treatment of the perturbation problem, which we address next.

\section{Analysis of the perturbed Weber system}
\label{sec:Weber-analysis}
Applying Proposition~\ref{prop:Schrodinger-Q-estimate} or Proposition~\ref{prop:ZS-Q-estimate} as applicable, in both special cases of the general system \eqref{eq:general-system} we have therefore arrived at an equivalent system having a universal form:
\begin{equation}
    \epsilon\frac{\dd\mathbf{u}}{\dd y} =\begin{bmatrix}0 & 1\\b^2+\frac{1}{4}y^2 +\epsilon^2 Q(y;\lambda) & 0\end{bmatrix}\mathbf{u},\quad |Q(y;\lambda)|\lesssim\frac{1}{1+y^2}.
    \label{eq:perturbed-Weber-system}
\end{equation}
This system can be viewed as a perturbation of the model system \eqref{eq:Weber-system}, for which a fundamental solution matrix is given by 
\begin{equation}
    \mathbf{U}_0(y;b,\epsilon):=\begin{bmatrix}
        U(\epsilon^{-1}b,\epsilon^{-1/2}y) & V(\epsilon^{-1}b,\epsilon^{-1/2}y)\\
        \epsilon^{1/2}U_z(\epsilon^{-1}b,\epsilon^{-1/2}y) & \epsilon^{1/2}V_z(\epsilon^{-1}b,\epsilon^{-1/2}y)
    \end{bmatrix},
\end{equation}  
involving the parabolic cylinder functions $U(a,z)$ and $V(a,z)$.  To remove the dominant terms in \eqref{eq:perturbed-Weber-system} we make the substitution
\begin{equation}
    \mathbf{u}(y)=\mathbf{U}_0(y;b,\epsilon)\mathbf{v}(y),
\end{equation}
leading to 
\begin{equation}
    \frac{\dd\mathbf{v}}{\dd y}=\epsilon\mathbf{U}_0(y;b,\epsilon)^{-1}\begin{bmatrix}0&0\\Q(y;\lambda) & 0\end{bmatrix}\mathbf{U}_0(y;b,\epsilon)\mathbf{v}.
\end{equation}
Since $\det(\mathbf{U}_0(y;b,\epsilon))=\epsilon^{1/2}\W[U,V]$, where $\W[U,V]$ denotes the Wronskian with value $\sqrt{2/\pi}$ according to \eqref{eq:Weber-Wronskian}, this system can be written as
\begin{equation}
    \frac{\dd\mathbf{v}}{\dd y} = \epsilon^{1/2}\sqrt{\frac{\pi}{2}}Q(y;\lambda)\mathbf{K}(y;\epsilon)\mathbf{v},
\label{eq:v-system}
\end{equation}
with coefficient matrix
\begin{equation}
    \mathbf{K}(y;\epsilon):=\begin{bmatrix}-U(\epsilon^{-1}b,\epsilon^{-1/2}y)V(\epsilon^{-1}b,\epsilon^{-1/2}y) & -V(\epsilon^{-1}b,\epsilon^{-1/2}y)^2\\
    U(\epsilon^{-1}b,\epsilon^{-1/2}y)^2 & U(\epsilon^{-1}b,\epsilon^{-1/2}y)V(\epsilon^{-1}b,\epsilon^{-1/2}y)\end{bmatrix}.
\end{equation}

Recall that for both special cases of \eqref{eq:general-system} the gauge matrix relating $\mathbf{w}$ and $\mathbf{u}$ has unit determinant, and that under the Langer transformation $y=g(x;\lambda)$ we have $x\to\pm\infty$ if and only if $y\to\pm\infty$.  Therefore, any nontrivial solution $\mathbf{w}=\mathbf{w}^\pm(x)$ of \eqref{eq:general-system} decaying to $\mathbf{0}$ as $x\to\pm\infty$ corresponds to a solution $\mathbf{u}=\mathbf{u}^\pm(y)$ of \eqref{eq:perturbed-Weber-system} decaying as $y\to\pm\infty$.  We now construct such a solution $\mathbf{u}^+(y)$.  

Since $U(a,z)$ rapidly decays while $V(a,z)$ rapidly grows as $z\to+\infty$, by choice of normalization constant, the vector function $\mathbf{v}=\mathbf{v}^+(y)$ corresponding to $\mathbf{u}=\mathbf{u}^+(y)$ decaying as $y\to+\infty$ should satisfy the boundary condition $\mathbf{v}^+(y)\to [1,0]^\top$ as $y\to+\infty$.  We build in this boundary condition by writing $\mathbf{v}^+(y)=[1,0]^\top + \mathbf{r}^+(y)$.  Assuming sufficiently rapid decay of $\mathbf{r}^+(y)$ as $y\to+\infty$, we substitute and integrate \eqref{eq:v-system} from a variable point $y\in\mathbb{R}$ to $+\infty$ and obtain
\begin{equation}
    \mathbf{r}^+(y)=-\epsilon^{1/2}\sqrt{\frac{\pi}{2}}\int_y^{+\infty}Q(\bar{y};\lambda)\mathbf{f}(\bar{y};\epsilon)\,\dd \bar{y} -\epsilon^{1/2}\sqrt{\frac{\pi}{2}}\int_y^{+\infty}Q(\bar{y};\lambda)\mathbf{K}(\bar{y};\epsilon)\mathbf{r}^+(\bar{y})\,\dd \bar{y},
\end{equation}
where $\mathbf{f}(y;\epsilon):=\mathbf{K}(y;\epsilon)[1,0]^\top$.  Since $K_{12}(y;\epsilon)$ grows rapidly as $y\to+\infty$, for this integral equation to make sense in $L^\infty(0,\infty)$ it is necessary to build in a corresponding rate of decay in the second component of $\mathbf{r}^+(y)$.  The correct rate is given either by the weight $W(\epsilon^{-1}b,\epsilon^{-1/2}y)^{-1}$ defined in \eqref{eq:PC-weight} or $Y(\epsilon^{-1}b,\epsilon^{-1/2}y)^{-1}$ defined in \eqref{eq:PC-weight-large-a}, depending respectively on whether $a=\epsilon^{-1}b<0$ is bounded or not (of size as large as $\epsilon^{-1}$) as $\epsilon\to 0$.  See Appendix~\ref{sec:PC-bounds}.  Fix $M>0$ sufficiently large.

\subsection{The case of low-lying eigenvalues:  $-\frac{1}{2}M^2\le a<0$}
Introducing the weight $W(a,z)$ defined by \eqref{eq:PC-weight}, we write $\mathbf{r}^+(y)=\mathrm{diag}(1,W(\epsilon^{-1}b,\epsilon^{-1/2}y)^{-1})\widetilde{\mathbf{r}}^+(y)$ and seek $\widetilde{\mathbf{r}}^+(y)$ in the space $L^\infty(0,\infty)$.  The corresponding integral equation satisfied by $\widetilde{\mathbf{r}}^+(y)$ reads
\begin{equation}
    \widetilde{\mathbf{r}}^+(y)=-\epsilon^{1/2}\sqrt{\frac{\pi}{2}}\int_y^{+\infty}Q(\bar{y};\lambda)\widetilde{\mathbf{f}}(y,\bar{y};\epsilon)\,\dd \bar{y} -\epsilon^{1/2}\sqrt{\frac{\pi}{2}}\int_y^{+\infty}Q(\bar{y};\lambda)\widetilde{\mathbf{K}}(y,\bar{y};\epsilon)\widetilde{\mathbf{r}}^+(\bar{y})\,\dd \bar{y},   
    \label{eq:zplus-tilde-low-lying-integral-equation}
\end{equation}
where with $a=\epsilon^{-1}b$ and $z=\epsilon^{-1/2}y$, $\bar{z}=\epsilon^{-1/2}\bar{y}$,
\begin{equation}
    \widetilde{\mathbf{f}}(y,\bar{y};\epsilon):=\begin{bmatrix}
        -U(a,\bar{z})V(a,\bar{z})\\W(a,z)U(a,\bar{z})^2
    \end{bmatrix},\quad\widetilde{\mathbf{K}}(y,\bar{y};\epsilon):=\begin{bmatrix}-U(a,\bar{z})V(a,\bar{z}) & -W(a,\bar{z})^{-1}V(a,\bar{z})^2\\
    W(a,z)U(a,\bar{z})^2 & W(a,z)W(a,\bar{z})^{-1}U(a,\bar{z})V(a,\bar{z})\end{bmatrix}.
\end{equation}
Since $y\le \bar{y}$ so also $z\le \bar{z}$, and since $W(a,z)$ is positive and nondecreasing in $z$, we can replace $W(a,z)$ with $W(a,\bar{z})$ and apply Lemma~\ref{lem:PC-bound-bounded-a} in the setting that $a=\epsilon^{-1}b<0$ is bounded below by a fixed value $-\frac{1}{2}M^2<0$ and that $\bar{y}\ge 0$ to see that both elements of $\widetilde{\mathbf{f}}(y,\bar{y};\epsilon)$ and all four elements of $\widetilde{\mathbf{K}}(y,\bar{y};\epsilon)$ are bounded in absolute value by a multiple of $(1+\bar{z}^2)^{-1/2}$.  Using the $L^\infty(0,\infty)$ norm $\sup_{y>0}\|\widetilde{\mathbf{r}}^+(y)\|$ subordinate to the pointwise norm on $\mathbb{C}^2$
\begin{equation}
    \|\mathbf{c}\|:=\max\{|c_1|,|c_2|\},\quad \text{implying}\quad \|\mathbf{Ac}\|\le \max\{|A_{11}|+|A_{12}|,|A_{21}|+|A_{22}|\}\|\mathbf{c}\|,
\end{equation}
we easily obtain the inequality
\begin{equation}
    \sup_{y>0}\|\widetilde{\mathbf{r}}^+(y)\|\lesssim\epsilon^{1/2}\sqrt{\frac{\pi}{2}}\int_0^{+\infty }\frac{\dd \bar{y}}{(1+\bar{y}^2)\sqrt{1+\epsilon^{-1}\bar{y}^2}}\left(1+\sup_{y>0}\|\widetilde{\mathbf{r}}^+(y)\|\right),
\end{equation}
where we have used $|Q(y;\lambda)|\lesssim (1+y^2)^{-1}$.
An estimate of the integral is the following:
\begin{equation}
    \int_0^{+\infty}\frac{\dd \bar{y}}{(1+\bar{y}^2)\sqrt{1+\epsilon^{-1}\bar{y}^2}}=\epsilon^{1/2}\int_0^{+\infty}\frac{\dd \bar{y}}{(1+\bar{y}^2)\sqrt{\epsilon + \bar{y}^2}}\lesssim\epsilon^{1/2}\log(\epsilon^{-1}),\quad\epsilon\to 0,
\end{equation}
as is easily seen by splitting up the integral at $\bar{y}=1$.  It therefore follows that there exists a unique solution of the integral equation \eqref{eq:zplus-tilde-low-lying-integral-equation} in $L^\infty(0,\infty)$ satisfying
$\sup_{y>0}\|\widetilde{\mathbf{r}}^+(y)\|\lesssim \epsilon\log(\epsilon^{-1})$.

We now show that the second component is smaller:  $\sup_{y>0}|\widetilde{r}^+_2(y)|\lesssim\epsilon$.  This comes from the fact that the second component of the forcing term satisfies
\begin{equation}
\begin{split}
\left|\int_y^{+\infty}Q(\bar{y};\lambda)\widetilde{f}_2(y,\bar{y};\epsilon)\,\dd \bar{y}\right|&\lesssim W(\epsilon^{-1}b,\epsilon^{-1/2}y)\int_y^{+\infty}|Q(\bar{y};\lambda)|U(\epsilon^{-1}b,\epsilon^{-1/2}\bar{y})^2\,\dd \bar{y}\\ &\lesssim W(\epsilon^{-1}b,\epsilon^{-1/2}y)\int_y^{+\infty}U(\epsilon^{-1}b,\epsilon^{-1/2}\bar{y})^2\,\dd \bar{y},
\end{split}
\end{equation}
where we used the estimate $|Q(y;\lambda)|\lesssim (1+y^2)^{-1}\lesssim 1$. If $0\le \epsilon^{-1/2}y\le M$, then $W(\epsilon^{-1}b,\epsilon^{-1/2}y) = \ee^{M^2/2}M^{2b/\epsilon}\lesssim 1$ since $b=O(\epsilon)$ in the low-lying case, so
\begin{equation}
    \left|\int_y^{+\infty}Q(\bar{y};\lambda)\widetilde{f}_2(y,\bar{y};\epsilon)\,\dd \bar{y}\right|\lesssim\int_0^{+\infty}U(\epsilon^{-1}b,\epsilon^{-1/2}\bar{y})^2\,\dd \bar{y} = 
    \epsilon^{1/2}\int_0^{+\infty}U(\epsilon^{-1}b,\bar{z})^2\,\dd \bar{z} = O(\epsilon^{1/2}),
\end{equation}
using again that $b=O(\epsilon)$.  On the other hand, if $\epsilon^{-1/2}y>M$, then also $\epsilon^{-1/2}\bar{y}>M$, so using \eqref{eq:PC-weight} and Lemma~\ref{lem:PC-bound-bounded-a} gives
\begin{equation}
\begin{split}
    \left|\int_y^{+\infty}Q(\bar{y};\lambda)\widetilde{f}_2(y,\bar{y};\epsilon)\,\dd \bar{y}\right|&\lesssim
    W(\epsilon^{-1}b,\epsilon^{-1/2}y)\int_y^{+\infty}\frac{\dd \bar{y}}{W(\epsilon^{-1}b,\epsilon^{-1/2}\bar{y})\sqrt{1+\epsilon^{-1} \bar{y}^2}}\\ &=\ee^{y^2/(2\epsilon)}\left(\frac{y^2}{\epsilon}\right)^{b/\epsilon}\int_y^{+\infty}\ee^{-\bar{y}^2/(2\epsilon)}\left(\frac{\bar{y}^2}{\epsilon}\right)^{-b/\epsilon}\frac{\dd \bar{y}}{\sqrt{1+\epsilon^{-1} \bar{y}^2}}\\
    &\lesssim \ee^{y^2/(2\epsilon)}\left(\frac{y^2}{\epsilon}\right)^{b/\epsilon}\int_y^{+\infty}\ee^{-\bar{y}^2/(2\epsilon)}\left(\frac{\bar{y}^2}{\epsilon}\right)^{-b/\epsilon-1/2}\,\dd \bar{y}\\
    &=\epsilon^{1/2}\ee^{y^2/(2\epsilon)}\left(\frac{y^2}{\epsilon}\right)^{b/\epsilon}\int_{\epsilon^{-1/2}y}^{+\infty}\ee^{-\bar{z}^2/2}\bar{z}^{-2b/\epsilon-1}\,\dd \bar{z}.
    \end{split}
\end{equation}
By repeated integration by parts, the resulting integral has a complete asymptotic expansion with respect to the sequence of gauge functions $\{\ee^{-z^2/2}z^{-2b/\epsilon-2-2k}\}_{k=0}^\infty$ as $z\to+\infty$, where $z=\epsilon^{-1/2}y$.  In particular, for $\epsilon^{-1/2}y>M$, we can estimate using just the leading term and obtain
\begin{equation}
    \left|\int_y^{+\infty}Q(\bar{y};\lambda)\widetilde{f}_2(y,\bar{y};\epsilon)\,\dd\bar{y}\right|\lesssim\epsilon^{1/2}\left(\frac{y^2}{\epsilon}\right)^{-1} =O(\epsilon^{1/2}).
\end{equation}
In both cases, we obtain a contribution to $\widetilde{r}_2^+(y)$ from the forcing term that is uniformly $O(\epsilon)$ for $y\ge 0$.  This dominates the contribution from the second row of the kernel $\widetilde{\mathbf{K}}(y,\bar{y};\epsilon)$; indeed iterating with the a priori estimate $\sup_{y>0}\|\widetilde{\mathbf{r}}^+(y)\|=O(\epsilon\log(\epsilon^{-1}))$ gives a substantially smaller contribution proportional to $\epsilon^2\log(\epsilon^{-1})^2$. 

From \eqref{eq:PC-weight} for $z=0$, we have $r^+_1(0)=\widetilde{r}^+_1(0)=O(\epsilon\log(\epsilon^{-1}))$ while $r^+_2(0)=W(a,0)^{-1}\widetilde{r}^+_2(0)=\ee^{-M^2/2}M^{-2a}\widetilde{r}_2^+(0) = O(\epsilon)$ because $0\le -2a\le M^2$, all estimates being uniform for the indicated range of $a$.

\subsection{The case of eigenvalues large compared to $\epsilon$:  $a\le -\frac{1}{2}M^2<0$}
In this complementary case, we introduce instead the alternate weight $Y(a,z)$ defined by \eqref{eq:PC-weight-large-a} via $\mathbf{r}^+(y)=\mathrm{diag}(1,Y(\epsilon^{-1}b,\epsilon^{-1/2}y)^{-1})\widetilde{\mathbf{r}}^+(y)$, obtaining again \eqref{eq:zplus-tilde-low-lying-integral-equation} with modified definitions for $\widetilde{\mathbf{f}}(y,\bar{y};\epsilon)$ and $\widetilde{\mathbf{K}}(y,\bar{y};\epsilon)$:
\begin{equation}
    \widetilde{\mathbf{f}}(y,\bar{y};\epsilon):=\begin{bmatrix}
        -U(a,\bar{z})V(a,\bar{z})\\Y(a,z)U(a,\bar{z})^2
    \end{bmatrix},\quad\widetilde{\mathbf{K}}(y,\bar{y};\epsilon):=\begin{bmatrix}-U(a,\bar{z})V(a,\bar{z}) & -Y(a,\bar{z})^{-1}V(a,\bar{z})^2\\
    Y(a,z)U(a,\bar{z})^2 & Y(a,z)Y(a,\bar{z})^{-1}U(a,\bar{z})V(a,\bar{z})\end{bmatrix}.
\end{equation}
The weight $Y(a,z)$ is still positive and nondecreasing in $z$, so we may follow similar reasoning as in the low-lying case.
Since now we have $a=\epsilon^{-1}b\le -\frac{1}{2}M^2<0$, we apply Lemma~\ref{lem:PC-bound-large-a} from Appendix~\ref{sec:PC-bounds} to obtain
\begin{equation}
    \sup_{y>0}\|\widetilde{\mathbf{r}}^+(y)\|\lesssim\epsilon^{1/2}\sqrt{\frac{\pi}{2}}\int_0^{+\infty}\frac{\dd \bar{y}}{(1+\bar{y}^2)(-a)^{1/2}(1+|\bar{t}|)^{1/3}|\zeta(\bar{t})|^{1/2}}\left(1+\sup_{y>0}\|\widetilde{\mathbf{r}}^+(y)\|\right),
\end{equation}
where $a=\epsilon^{-1}b$, $\bar{t}=\bar{z}/\sqrt{-4a}=\bar{y}/\sqrt{-4b}$, and $\zeta(t)$ is the increasing function of $t$ defined in Appendix~\ref{sec:zeta-t-properties} with $\zeta(1)=0$ and $\zeta(t)=(\frac{3}{4})^{2/3}t^{4/3}(1+O(t^{-1}))$ as $t\to+\infty$.  By changing the integration variable to $\bar{t}$ we get
\begin{equation}
\begin{split}
    \int_0^{+\infty}\frac{\dd \bar{y}}{(1+\bar{y}^2)(-a)^{1/2}(1+|\bar{t}|)^{1/3}|\zeta(\bar{t})|^{1/2}}&=2\epsilon^{1/2}\int_0^{+\infty}\frac{\dd \bar{t}}{(1-4b\bar{t}^2)(1+|\bar{t}|)^{1/3}|\zeta(\bar{t})|^{1/2}}\\ &\lesssim \epsilon^{1/2}\log(\ee-b^{-1}),
\end{split}
\label{eq:integral-a-large}
\end{equation}
where the last estimate can be seen by splitting the integral at $\bar{t}=2$ and using $(1+\bar{t})^{-1/3}|\zeta(\bar{t})|^{-1/2}\lesssim 1/\bar{t}$ for $\bar{t}\ge 2$.  In other words, the integral in \eqref{eq:integral-a-large} is proportional to $\epsilon^{1/2}$ when $b<0$ is bounded away from zero but more generally is proportional to $\epsilon^{1/2}\log(-b^{-1})$ when $b<0$ becomes small.  Since the lower bound for $-b$ under the assumption $a\le -\frac{1}{2}M^2<0$ is proportional to $\epsilon$, we conclude that the estimate $\sup_{y>0}\|\widetilde{\mathbf{r}}^+(y)\|\lesssim \epsilon\log(\epsilon^{-1})$ holds in this case exactly as in the case of low-lying eigenvalues.

It is also true that when $a\le -\frac{1}{2}M^2<0$, we have the sharper estimate $\sup_{y>0}|\widetilde{r}^+_2(y)|\lesssim\epsilon$, just as in the low-lying case.  For the same reason as in that case the claim reduces to the analysis of the second component of the forcing function, which now involves the weight function $Y(a,z)$ defined in \eqref{eq:PC-weight-large-a}:
\begin{equation}
\begin{split}
\left|\int_y^{+\infty}Q(\bar{y};\lambda)\widetilde{f}_2(y,\bar{y};\epsilon)\,\dd \bar{y}\right| &\lesssim Y(a,z)\int_y^{+\infty}|Q(\bar{y};\lambda)|U(a,\bar{z})^2\,\dd \bar{y}\\
    &\lesssim Y(a,z)\int_y^{+\infty}U(a,\bar{z})^2\,\dd \bar{y}\\
    &\lesssim Y(a,z)\int_y^{+\infty}\frac{\dd \bar{y}}{Y(a,\bar{z})(-a)^{1/2}(1+|\bar{t}|)^{1/3}|\zeta(\bar{t})|^{1/2}},
    \end{split}
    \label{eq:before-simplify-Y}
\end{equation}
where as before $a=\epsilon^{-1}b$, $z=\epsilon^{-1/2}y$, $\bar{t}=\bar{y}/\sqrt{-4b}$, and we used $|Q(y;\lambda)|\lesssim (1+y^2)^{-1}\lesssim 1$ and Lemma~\ref{lem:PC-bound-large-a}.  If $0\le y\le \sqrt{-4b}$, then $Y(a,z)/Y(a,\bar{z})=\ee^{8a\zeta_+(\bar{t})^{3/2}/3}$, and after changing variables from $\bar{y}$ to $\bar{t}$,
\begin{equation}
\begin{split}
\left|\int_y^{+\infty}Q(\bar{y};\lambda)\widetilde{f}_2(y,\bar{y};\epsilon)\,\dd y'\right| &\lesssim 2\epsilon^{1/2}\int_{y/\sqrt{-4b}}^{+\infty}\frac{\ee^{8a\zeta_+(\bar{t})^{3/2}/3}\,\dd \bar{t}}{(1+\bar{t})^{1/3}|\zeta(\bar{t})|^{1/2}}\\
&\le 2\epsilon^{1/2}\int_0^{+\infty}\frac{\ee^{-4M^2\zeta_+(\bar{t})^{3/2}/3}\,\dd \bar{t}}{(1+\bar{t})^{1/3}|\zeta(\bar{t})|^{1/2}},
\end{split}
\end{equation}
where we used $a\le -\frac{1}{2}M^2$.  Since the last integral is convergent and depends on $M$ only, this estimate is $O(\epsilon^{1/2})$.  A similar computation extends the same estimate to hold for $y>\sqrt{-4b}$ (i.e., for $t>1$ and hence also $\zeta(t)>0$), provided that $-\frac{8}{3}a\zeta(t)^{3/2}\le K$ for some $K>0$ fixed, since we need only start instead from \eqref{eq:before-simplify-Y} using $Y(a,z)/Y(a,\bar{z})\le \ee^K\ee^{8a\zeta(\bar{t})^{3/2}/3}$ ($\zeta_+(\bar{t})=\zeta(\bar{t})\ge \zeta(t)>0$ in this case), so the implied constant in $O(\epsilon^{1/2})$ is modified by a fixed positive factor $\ee^K$.  On the other hand, if $-\frac{8}{3}a\zeta(t)^{3/2}>K>0$, then introducing $k=\frac{8}{3}\zeta(t)^{3/2}>0$ and corresponding integration variable $\bar{k}=-\frac{8}{3}\zeta(\bar{t})^{3/2}\ge k$, 
\begin{equation}
    \begin{split}
\left|\int_y^{+\infty}Q(\bar{y};\lambda)\widetilde{f}_2(y,\bar{y};\epsilon)\,\dd \bar{y}\right| &\lesssim 2\epsilon^{1/2}\ee^{-8a\zeta(t)^{3/2}/3}\int_{y/\sqrt{-4b}}^{+\infty}\frac{\ee^{8a\zeta(\bar{t})^{3/2}/3}\,\dd \bar{t}}{(1+\bar{t})^{1/3}\zeta(\bar{t})^{1/2}}   \\
&=-\frac{\epsilon^{1/2}}{3^{1/3}}\ee^{-ak}\int_{k}^{+\infty}\frac{\ee^{-a\bar{k}}\,\dd\bar{k}}{\bar{k}^{1/3}(\bar{t}^2-1)^{1/2}(1+\bar{t})^{1/3}},
    \end{split}
\end{equation}
where we used the definition of $\zeta(t)$ in Appendix~\ref{sec:zeta-t-properties} to obtain $\dd \bar{t}/\dd \bar{k}=\frac{1}{4}\zeta(\bar{t})^{-1/2}\zeta'(\bar{t})^{-1}$ and $\zeta(\bar{t})^{1/2}\zeta'(\bar{t})=(\bar{t}^2-1)^{1/2}$, and finally eliminated $\zeta(\bar{t})$ in favor of $\bar{k}$.  Then one combines the asymptotic behavior of $\zeta(\bar{t})$ for $\bar{t}-1$ small and $\bar{t}$ large with the definition of $\bar{k}$ to see that over the whole range of integration, $(\bar{t}^2-1)^{1/2}\gtrsim\bar{k}^{1/3}$ while $(1+\bar{t})^{1/3}\gtrsim \bar{k}^{1/6}$ and therefore
\begin{equation}
    \begin{split}
      \left|\int_y^{+\infty}Q(\bar{y};\lambda)\widetilde{f}_2(y,\bar{y};\epsilon)\,\dd \bar{y}\right| &\lesssim  \epsilon^{1/2}\ee^{-ak}\int_k^{+\infty}\frac{\ee^{a\bar{k}}\,\dd \bar{k}}{\bar{k}^{5/6}}\\
      &\le \frac{\epsilon^{1/2}}{k^{5/6}}\ee^{-ak}\int_k^{+\infty}\ee^{a\bar{k}}\,\dd \bar{k}\\
      &= \frac{\epsilon^{1/2}}{(-a)k^{5/6}}.
    \end{split}
\end{equation}
This estimate holds under the assumption that $-ak>K>0$ or $k^{5/6}>K^{5/6}(-a)^{-5/6}$; hence 
\begin{equation}
    \left|\int_y^{+\infty}Q(\bar{y};\lambda)\widetilde{f}_2(y,\bar{y};\epsilon)\,\dd \bar{y}\right| \lesssim \epsilon^{1/2}(-a)^{-1/6}
\end{equation}  
which is less than or equal to $\epsilon^{1/2}(\frac{1}{2}M^2)^{-1/6}=O(\epsilon^{1/2})$ because $a\le -\frac{1}{2}M^2$.
Combining these results gives a dominant contribution to $|\widetilde{r}_2^+(y)|$ of order $O(\epsilon)$.
Therefore, using \eqref{eq:PC-weight-large-a} for $z=0$ gives $r_1^+(0)=\widetilde{r}_1^+(0)=O(\epsilon\log(\epsilon^{-1}))$ while $r_2^+(0)=Y(a,0)^{-1}\widetilde{r}_2^+(0)=O(\ee^a(-a)^{-a}\epsilon)$.  

\subsection{The solution $\mathbf{u}^-(y)$}
To obtain a nontrivial solution $\mathbf{u}^-(y)$ of \eqref{eq:perturbed-Weber-system} decaying as $y\to -\infty$, we note that the function $\widehat{\mathbf{u}}(y):=\sigma_3\mathbf{u}^+(-y)$ is a solution of the closely-related system
\begin{equation}
    \epsilon\frac{\dd\widehat{\mathbf{u}}}{\dd y}=\begin{bmatrix}0 & 1\\b^2+\frac{1}{4}y^2+\epsilon^2 Q(-y;\lambda) & 0\end{bmatrix}\widehat{\mathbf{u}}
\end{equation}
which differs from \eqref{eq:perturbed-Weber-system} only in the replacement of $Q(y;\lambda)$ by $Q(-y;\lambda)$.  However, as our analysis only used the property $|Q(y;\lambda)|\lesssim 1/(1+y^2)$ which is obviously invariant under $y\mapsto -y$, we can easily adapt the estimates obtained for $\mathbf{u}^+(y)$ to obtain estimates on a solution $\mathbf{u}^-(y)$ decaying as $y\to -\infty$ of exactly the same perturbed system \eqref{eq:perturbed-Weber-system}.  Indeed, writing $\mathbf{u}^-(y)=\sigma_3\mathbf{U}_0(-y;b,\epsilon)\mathbf{v}^-(y)$ with $\mathbf{v}^-(y)=[1,0]^\top + \mathbf{r}^-(y)$ and $r^-_1(y)=\widetilde{r}^-_1(y)$, and either $r^-_2(y)=W(\epsilon^{-1}b,-\epsilon^{-1/2}y)^{-1}\widetilde{r}_2(y)$ or $r^-_2(y)=Y(\epsilon^{-1}b,-\epsilon^{-1/2}y)^{-1}\widetilde{r}_2^-(y)$ depending on whether $a\ge -\frac{1}{2}M^2$ or $a\le -\frac{1}{2}M^2$, we obtain exactly the same estimates:  $\widetilde{r}_1^-(y)=O(\epsilon\log(\epsilon^{-1}))$ and $\widetilde{r}_2^-(y)=O(\epsilon)$ hold uniformly for $a\le 0$ and $y\le 0$.

\subsection{Wronskian computation}
We now compute the Wronskian $\W[\mathbf{u}^+,\mathbf{u}^-](y)$ at $y=0$.  Using $\mathbf{u}^+(0)=\mathbf{U}_0(0;b,\epsilon)\mathbf{v}^+(0)$ and $\mathbf{u}^-(0)=\sigma_3\mathbf{U}_0(0;b,\epsilon)\mathbf{v}^-(0)$, this is given by 
\begin{equation}
\begin{split}
    \W[\mathbf{u}^+,\mathbf{u}^-](0)&=\begin{vmatrix}
        v^+_1(0)U(a,0) + v^+_2(0)V(a,0) & v^-_1(0)U(a,0)+ v^-_2(0)V(a,0)\\v_1^+(0)\epsilon^{1/2}U_z(a,0) + v_2^+(0)\epsilon^{1/2}V_z(a,0) & -v_1^-(0)\epsilon^{1/2}U_z(a,0)-v_2^-(0)\epsilon^{1/2}V_z(a,0)
    \end{vmatrix}
    \\
    &=
-\epsilon^{1/2}v_1^+(0)v_1^-(0)\left[2U(a,0)U_z(a,0) + \frac{v_2^+(0)v_2^-(0)}{v_1^+(0)v_1^-(0)}2V(a,0)V_z(a,0)\right.\\
&\qquad{}\left.+ \left(\frac{v_2^+(0)}{v_1^+(0)}+\frac{v_2^-(0)}{v_1^-(0)}\right)(U(a,0)V_z(a,0)+U_z(a,0)V(a,0))\right]
    \end{split}
\end{equation}
The explicit terms involving parabolic cylinder functions at $z=0$ have known values \cite[12.2 (ii)]{DLMF}; after some simplification, these can be written as 
\begin{equation}
\begin{split}
    2U(a,0)U_z(a,0)&=-\sqrt{\frac{2}{\pi}}\Gamma(\tfrac{1}{2}-a)\cos(-\pi a),\\
    2V(a,0)V_z(a,0)&=\sqrt{\frac{2}{\pi}}\frac{\cos(-\pi a)}{\Gamma(\frac{1}{2}-a)},\\ U(a,0)V_z(a,0)+U_z(a,0)V(a,0)&=\sqrt{\frac{2}{\pi}}\sin(-\pi a).
\end{split}
\end{equation}
For $a<0$, we may therefore write
\begin{equation}
    \sqrt{\frac{\pi}{2}}\frac{\W[\mathbf{u}^+,\mathbf{u}^-](0)}{\epsilon^{1/2}v_1^+(0)v_1^-(0)\Gamma(\frac{1}{2}-a)} = \cos(-\pi a)-\frac{v_2^+(0)v_2^-(0)}{v_1^+(0)v_1^-(0)}\frac{\cos(-\pi a)}{\Gamma(\frac{1}{2}-a)^2}-\left(\frac{v_2^+(0)}{v_1^+(0)}+\frac{v_2^-(0)}{v_1^-(0)}\right)\frac{\sin(-\pi a)}{\Gamma(\frac{1}{2}-a)}.
\end{equation}
Now, over the whole range $a<0$ we have $v_1^\pm(0)=1+r_1^\pm(0)=1+\widetilde{r}_1^\pm(0) = 1+O(\epsilon\log(\epsilon^{-1}))$ as $\epsilon\to 0$; in particular these quantities are nonvanishing for sufficiently small $\epsilon>0$ which justifies division by the product $v_1^+(0)v_1^-(0)$ (and likewise division by $\Gamma(\frac{1}{2}-a)>0$ is justified for $a<0$).  

In the case of low-lying eigenvalues, $-\frac{1}{2}M^2\le a<0$, we use $v_2^\pm(0)=r_2^\pm(0)=W(a,0)^{-1}\widetilde{r}_2^\pm(0)=O(\epsilon)$ and the fact that $\Gamma(\frac{1}{2}-a)\ge\Gamma(\frac{1}{2})>0$ to immediately obtain the result that 
\begin{equation}
        \sqrt{\frac{\pi}{2}}\frac{\W[\mathbf{u}^+,\mathbf{u}^-](0)}{\epsilon^{1/2}v_1^+(0)v_1^-(0)\Gamma(\frac{1}{2}-a)} = \cos(-\pi a) + O(\epsilon).
\label{eq:Wronskian-final}
\end{equation}
In the complementary case that $a\le -\frac{1}{2}M^2<0$, we have instead that $v_2^\pm(0)=r_2^\pm(0)=Y(a,0)^{-1}\widetilde{r}_2^\pm(0) = O(\ee^a(-a)^{-a}\epsilon)$.  But by Stirling's formula, we also have $\Gamma(\frac{1}{2}-a)\gtrsim \ee^a(-a)^{-a}$, so exactly the same result \eqref{eq:Wronskian-final} holds in this case as well.

Since $\lambda$ is an eigenvalue if and only if $\W[\mathbf{u}^+,\mathbf{u}^-](0)=0$, using $a=b/\epsilon$ and substituting for $b=b(\lambda)$ by \eqref{eq:b-of-lambda} completes the proofs of Theorems~\ref{thm:Schroedinger-BS} and \ref{thm:ZS-BS}.

\appendix

\appendixpage

\section{The functions $\zeta(t)$ and $t(\zeta)$}
\label{sec:zeta-t-properties}
Throughout our work, an important role is played by the function $\zeta:(-1,+\infty)\to\mathbb{R}$ defined for $t>1$ by 
\begin{equation}
    \zeta(t):=\left(\frac{3}{2}\int_1^{t}\sqrt{\tau^2-1}\,\dd\tau\right)^{2/3},\quad t>1
    \label{eq:zeta-define}
\end{equation}
and continued analytically to the half-line $t>-1$ as a Schwarz-symmetrical function. An explicit formula for $\zeta$ on the interval $-1<t<1$ is:
\begin{equation}
    \zeta(t)=-\left(\frac{3}{2}\int_t^1\sqrt{1-\tau^2}\,\dd\tau\right)^{2/3},\quad -1<t<1.
\label{eq:zeta-less-than-1}
\end{equation}
We have the asymptotic behavior
\begin{equation}
    \zeta(t) = \left(\frac{3}{4}\right)^{2/3}t^{4/3}(1+O(t^{-1})),\quad t\to+\infty.
    \label{eq:zeta-asymp}
\end{equation}
Note that $\zeta(t)$ vanishes linearly at $t=1$, and it is strictly increasing on $(-1,+\infty)$, so it has an analytic inverse denoted $t:(\zeta(-1),+\infty)\to (-1,+\infty)$.  It is straightforward to show that $t(\diamond)$ and its derivatives have the following asymptotic behavior:
    \begin{equation}
    \begin{gathered}
        t(\zeta)=\left(\frac{4}{3}\right)^{1/2}\zeta^{3/4}(1+O(\zeta^{-1})),\quad
t'(\zeta)=\left(\frac{3}{4}\right)^{1/2}\zeta^{-1/4}(1+O(\zeta^{-1})),\\ t''(\zeta)=-\frac{1}{4}\left(\frac{3}{4}\right)^{1/2}\zeta^{-5/4}(1+O(\zeta^{-1})),\quad t'''(\zeta)=\frac{5}{16}\left(\frac{3}{4}\right)^{1/2}\zeta^{-9/4}(1+O(\zeta^{-1})).
\end{gathered}
\label{eq:t-large-z}
    \end{equation}
From \eqref{eq:zeta-define}--\eqref{eq:zeta-less-than-1} we can easily derive the differential equation
\begin{equation}
    \left(\frac{\dd t}{\dd\zeta}\right)^2(t^2-1)=\zeta.
    \label{eq:t-zeta-ODE}
\end{equation}

\section{Estimates of parabolic cylinder functions}
\label{sec:PC-bounds}
Here we develop some simple estimates for quadratic expressions in the parabolic cylinder functions $U(a,z)$ and $V(a,z)$.  These estimates take different forms depending on whether or not the nonnegative parameter $a\le 0$ is bounded or unbounded.
\subsection{Estimates of $U(a,z)$ and $V(a,z)$ for $a$ bounded}
The parabolic cylinder functions $U(a,z)$ and $V(a,z)$ are jointly entire in $(a,z)\in\mathbb{C}^2$.  Therefore, both functions are bounded on any bidisk $\{|a|<r_a\}\times\{|z|<r_z\}$.  The asymptotic formul\ae\ \eqref{eq:UV-asymp} follow from steepest descent analysis of contour integral representations of $U(a,z)$ and $V(a,z)$, and one can use this analysis to verify that the $O(z^{-2})$ error terms are uniform for bounded $a$.  Therefore if we restrict $a$ to a bounded interval of negative values, say $-\frac{1}{2}M^2\le a\le 0$, and $z$ to nonnegative values $z\ge 0$, then we may bound the real-valued functions $U(a,z)^2$, $V(a,z)^2$, and $U(a,z)V(a,z)$ by a constant on any given finite interval $0\le z\le N^2$, but for unbounded $z>0$ we get bounds instead from the large-$z$ asymptotic \eqref{eq:UV-asymp}.  We therefore define a weight function by setting
\begin{equation}
    W(a,z):=\begin{cases}\left(e^{z^2/4}z^a\right)^2,&\quad z>M\\
    \left(\ee^{M^2/4}M^a\right)^2,&\quad0\le z\le M,
    \end{cases}
\label{eq:PC-weight}
\end{equation}
which has the property that it is positive and nondecreasing as a function of $z\ge 0$.  Then combining a uniform bound for bounded $z$ with the asymptotic behavior \eqref{eq:UV-asymp} proves the following.
\begin{lemma}[Estimates for weighted quadratic expressions in $U(a,z)$ and $V(a,z)$ valid for bounded $a<0$]
Fix $M>0$ and define a weight function $W(a,z)$ by \eqref{eq:PC-weight}.  Then the estimates
\begin{equation}
        U(a,z)^2W(a,z)\lesssim\frac{1}{\sqrt{1+z^2}},\quad 
        V(a,z)^2W(a,z)^{-1}\lesssim\frac{1}{\sqrt{1+z^2}},\quad
        |U(a,z)V(a,z)|\lesssim\frac{1}{\sqrt{1+z^2}}
\end{equation}
hold for $-\frac{1}{2}M^2\le a\le 0$ and $z\ge 0$ (the implied constant depends on $M$).
\label{lem:PC-bound-bounded-a}
\end{lemma}

\subsection{Estimates of $U(a,z)$ and $V(a,z)$ for $a$ large}
Here, we follow \cite[\S 12.10(vii)]{DLMF}, in which for $a<0$, the functions $U(a,z)$ and $V(a,z)$ are written in terms of a rescaled variable $t:=z/\sqrt{-4a}$ exactly in the forms
\begin{equation}
\begin{split}
    U(a,z)&=2\sqrt{\pi}(-2a)^{1/6}G(a)\phi(t)\left[\mathrm{Ai}((-2a)^{2/3}\zeta(t))A^U(t;a) + (-2a)^{-4/3}\mathrm{Ai}'((-2a)^{2/3}\zeta(t))B^U(t;a)\right],\\
    V(a,z)&=2\sqrt{\pi}(-2a)^{1/6}H(a)\phi(t)\left[\mathrm{Bi}((-2a)^{2/3}\zeta(t))A^V(t;a) + (-2a)^{-4/3}\mathrm{Bi}'((-2a)^{2/3}\zeta(t))B^V(t;a)\right],
\end{split}
\label{eq:UV-large-a-exact}
\end{equation}
where $G(a)$ is a well-defined function having asymptotic behavior 
\begin{equation}
    G(a)=\frac{1}{\sqrt{2}}\ee^{a/2}(-a)^{-a/2-1/4}(1+O(a^{-1})),\quad a\to -\infty,
\label{eq:G-expansion}
\end{equation}
$H(a):=G(a)/\Gamma(\frac{1}{2}-a)$ has corresponding behavior
\begin{equation}
    H(a)=\frac{1}{2\sqrt{\pi}}\ee^{-a/2}(-a)^{a/2-1/4}(1+O(a^{-1})),\quad a\to -\infty,
\label{eq:H-expansion}
\end{equation}
$\zeta(t)$ is defined in Appendix~\ref{sec:zeta-t-properties}, and 
\begin{equation}
    \phi(t):=\left(\frac{\zeta(t)}{t^2-1}\right)^{1/4},\quad t>-1,
\end{equation}
where the positive fourth root is meant (of a strictly positive argument).  The coefficients $A^U(t;a)$ and $A^V(t;a)$ are well-defined functions having a common asymptotic expansion:
\begin{equation}
    A^{U,V}(t;a)\sim\sum_{s=0}^\infty A_s(t)(-2a)^{-2s},\quad a\to -\infty,
\label{eq:A-series}
\end{equation}
and likewise the coefficients $B^U(t;a)$ and $B^V(t;a)$ are well-defined and have the same asymptotic expansion:
\begin{equation}
    B^{U,V}(t;a)\sim\sum_{s=0}^\infty B_s(t)(-2a)^{-2s},\quad a\to -\infty,
\label{eq:B-series}
\end{equation}
with all four expansions holding uniformly for $t\ge -1+\delta$ for fixed small $\delta>0$.  The expansion coefficients have the forms
\begin{equation}
\begin{split}   A_s(t)&:=\zeta(t)^{-3s}\sum_{m=0}^{2s}\beta_m\phi(t)^{6(2s-m)}u_{2s-m}(t)\\
B_s(t)&=-\zeta(t)^{-3s-2}\sum_{m=0}^{2s+1}\alpha_m\phi(t)^{6(2s-m+1)}u_{2s-m+1}(t)
\end{split}
\end{equation}
in which $\alpha_m$ and $\beta_m$ are numerical coefficients defined in \cite[12.10.43]{DLMF} the specific values of which we will not need, and $u_k(t)$ is a well-defined polynomial in $t$ with $u_0(t)=1$ and otherwise $u_k(t)$ has degree $3k$ for $k$ odd and $3k-2$ for $k\ge 2$ even.

Since $\zeta(t)$ and $\phi(t)$ are analytic functions for $t>-1$, they are clearly bounded uniformly on compact subsets, and it is easy to see from \eqref{eq:zeta-asymp} that $\phi(t)=(\frac{3}{4})^{1/6}t^{-1/6}(1+O(t^{-1}))$ as $t\to+\infty$.  It is then straightforward to check that for each $s=0,1,2,\dots$, there is a constant $C_s>0$ such that 
\begin{equation}
  |A_s(t)|\le C_s\quad\text{and}\quad|B_s(t)|\le   \frac{C_s}{(1+|t|)^{2/3}},\quad t\ge -1+\delta.
\end{equation}
Hence \eqref{eq:A-series} implies that for $M>0$ sufficiently large,
\begin{equation}
    |A^U(t;a)|\le 2C_0\quad\text{and}\quad |A^V(t;a)|\le 2C_0.
\end{equation}
both hold if $a\le -\frac{1}{2}M^2$ and $t\ge -1+\delta$.
Similarly, \eqref{eq:B-series} implies that for $M>0$ sufficiently large,
\begin{equation}
    |B^U(t;a)|\le \frac{2C_0}{(1+|t|)^{2/3}}\quad\text{and}\quad |B^V(t;a)|\le\frac{2C_0}{(1+|t|)^{2/3}}
\end{equation}
both hold under the same conditions.  Since also $|\phi(t)|\lesssim(1+|t|)^{-1/6}$ for $t\ge -1+\delta$, while \eqref{eq:G-expansion} and \eqref{eq:H-expansion} imply respectively that if $M>0$ is sufficiently large,
\begin{equation}
    \left|(-2a)^{1/6}G(a)\right|\le  (-2a)^{1/6}\cdot 2\frac{1}{\sqrt{2}}\ee^{a/2}(-a)^{-a/2-1/4}=2^{2/3}\ee^{a/2}(-a)^{-a/2-1/12}
\end{equation}
and
\begin{equation}
    \left|(-2a)^{1/6}H(a)\right|\le (-2a)^{1/6}\cdot 2\frac{1}{2\sqrt{\pi}}\ee^{-a/2}(-a)^{a/2-1/4} = \frac{2^{1/6}}{\sqrt{\pi}}\ee^{-a/2}(-a)^{a/2-1/12}
\end{equation}
both hold for $a\le -\frac{1}{2}M^2$, from \eqref{eq:UV-large-a-exact} we obtain
\begin{equation}
\begin{split}
    |U(a,z)|&\lesssim \frac{\ee^{a/2}(-a)^{-a/2-1/12}}{(1+|t|)^{1/6}}\left[\left|\mathrm{Ai}((-2a)^{2/3}\zeta(t))\right| + \frac{(-2a)^{-4/3}}{(1+|t|)^{2/3}}\left|\mathrm{Ai}'((-2a)^{2/3}\zeta(t))\right|\right]\\
    |V(a,z)|&\lesssim \frac{\ee^{-a/2}(-a)^{a/2-1/12}}{(1+|t|)^{1/6}}\left[\left|\mathrm{Bi}((-2a)^{2/3}\zeta(t))\right| + \frac{(-2a)^{-4/3}}{(1+|t|)^{2/3}}\left|\mathrm{Bi}'((-2a)^{2/3}\zeta(t))\right|\right]
\end{split}
\end{equation}
valid for $a\le -\frac{1}{2}M^2$ and $t\ge -1+\delta$.  

Finally, we can apply the estimates
\begin{equation}
    |\mathrm{Ai}(w)|\lesssim\frac{\ee^{-2w^{3/2}_+/3}}{|w|^{1/4}},\quad |\mathrm{Ai}'(w)|\lesssim |w|^{1/4}\ee^{-2w_+^{3/2}/3},\quad|\mathrm{Bi}(w)|\lesssim\frac{\ee^{2w_+^{3/2}/3}}{|w|^{1/4}},\quad|\mathrm{Bi}'(w)|\lesssim|w|^{1/4}\ee^{2w_+^{3/2}/3}
\end{equation}
that hold for all $w\in\mathbb{R}$, where $w_+$ denotes the positive part of $w$, i.e., $w_+=w$ for $w>0$ and $w_+=0$ otherwise.  Thus,
\begin{equation}
    \begin{split}
        |U(a,z)|&\lesssim\frac{\ee^{a/2}(-a)^{-a/2-1/12}}{(1+|t|)^{1/6}}\frac{\ee^{4a\zeta_+(t)^{3/2}/3}}{(-a)^{1/6}|\zeta(t)|^{1/4}}\left[1+(-2a)^{-1}\frac{|\zeta(t)|^{1/2}}{(1+|t|)^{2/3}}\right],\\
        |V(a,z)|&\lesssim\frac{\ee^{-a/2}(-a)^{a/2-1/12}}{(1+|t|)^{1/6}}\frac{\ee^{-4a\zeta_+(t)^{3/2}/3}}{(-a)^{1/6}|\zeta(t)|^{1/4}}\left[1+(-2a)^{-1}\frac{|\zeta(t)|^{1/2}}{(1+|t|)^{2/3}}\right].
    \end{split}
\end{equation}
Since $|\zeta(t)|^{1/2}/(1+|t|)^{2/3}$ is uniformly bounded for $t\ge -1+\delta$ and $a\le -\frac{1}{2}M^2$, we obtain simply
\begin{equation}
        |U(a,z)|\lesssim\frac{\ee^{a/2}(-a)^{-a/2-1/4}\ee^{4a\zeta_+(t)^{3/2}/3}}{(1+|t|)^{1/6}|\zeta(t)|^{1/4}}\quad\text{and}\quad
|V(a,z)|\lesssim\frac{\ee^{-a/2}(-a)^{a/2-1/4}\ee^{-4a\zeta_+(t)^{3/2}/3}}{(1+|t|)^{1/6}|\zeta(t)|^{1/4}}
\end{equation}
both valid for $a\le -\frac{1}{2}M^2$ and $t\ge -1+\delta$.  These imply the following.
\begin{lemma}[Estimates for weighted quadratic expressions in $U(a,z)$ and $V(a,z)$ valid for large $a<0$]
Writing $t:=z/\sqrt{-4a}$ for $a<0$, let a weight function be defined for $t>-1$ by
\begin{equation}
    Y(a,z):=\ee^{-a}(-a)^{a}\ee^{-8a\zeta_+(t)^{3/2}/3}=\begin{cases}\ee^{-a}(-a)^a,& t\le 1,\\
    \ee^{-a}(-a)^a\ee^{-8a\zeta(t)^{3/2}/3},&t>1,
    \end{cases}
    \label{eq:PC-weight-large-a}
\end{equation}
where $\zeta(t)$ is an analytic function of $t>-1$ defined by \eqref{eq:zeta-define}, strictly increasing, with $\zeta(1)=0$ and $\zeta(t)=(\frac{3}{4})^{2/3}t^{4/3}(1+O(t^{-1}))$ as $t\to+\infty$.  Then for $M>0$ sufficiently large and $a\le -\frac{1}{2}M^2$ and $t\ge -1+\delta$, the following estimates hold:
\begin{equation}
\begin{split}
    U(a,z)^2Y(a,z)&\lesssim \frac{1}{(-a)^{1/2}(1+|t|)^{1/3}|\zeta(t)|^{1/2}},\\ V(a,z)^2Y(a,z)^{-1}&\lesssim \frac{1}{(-a)^{1/2}(1+|t|)^{1/3}|\zeta(t)|^{1/2}},\\|U(a,z)V(a,z)|&\lesssim \frac{1}{(-a)^{1/2}(1+|t|)^{1/3}|\zeta(t)|^{1/2}}.
\end{split}
\end{equation}
\label{lem:PC-bound-large-a}
\end{lemma}

\end{document}